\documentclass{scrartcl}

\usepackage{amsmath,amsthm,amssymb,dsfont,booktabs}
\numberwithin{equation}{section}
\usepackage{mathtools,units}
\usepackage{graphicx,placeins}
\usepackage{xcolor}
\usepackage{multirow}
\usepackage{subcaption}
\usepackage[colorlinks=true]{hyperref}
\hypersetup{
  linkcolor=red!60!black,
  citecolor=green!60!black,
}

\newcommand{\X}{{\mathcal X}}
\newcommand{\Y}{{\mathcal Y}}
\newcommand{\V}{{\mathcal V}}
\newcommand{\W}{{\mathcal W}}

\newtheorem{lemma}{Lemma}
\newtheorem{theorem}[lemma]{Theorem}
\newtheorem{remark}[lemma]{Remark}
\newtheorem{algo}[lemma]{Algorithm}
\newtheorem{problem}[lemma]{Problem}
\newtheorem{assumption}[lemma]{Assumption}

\author{L. Lautsch\thanks{Otto-von-Guericke Universit\"at Magdeburg,
    Germany \texttt{leopold.lautsch@ovgu.de}} \and
  T. Richter\thanks{Otto-von-Guericke Universit\"at Magdeburg,
    Germany \texttt{thomas.richter@ovgu.de}}}

\title{Error estimation and adaptivity for differential equations with
  multiple scales in time}

\newcommand{\EA}{\textbf{\textsf{(EA)}}}
\newcommand{\EG}{\textbf{\textsf{(EG)}}}

\newcommand{\EF}{\textbf{\textsf{(EF)}}}
\newcommand{\ED}{\textbf{\textsf{(ED)}}}

\newcommand{\Ea}{\textbf{\textsf{(Ea)}}}
\newcommand{\Eb}{\textbf{\textsf{(Eb)}}}

\newcommand{\xt}{\mathbf{x}}

\newcommand\Cas[1]{C_{\hyperref[ass:#1]{A\ref*{ass:#1}}}}
\newcommand\Las[1]{\lambda_{\hyperref[ass:#1]{A\ref*{ass:#1}}}}

\begin{document}

\maketitle
\begin{abstract}
  We consider systems of ordinary differential equations with multiple
  scales in time. In general, we are interested in the long time
  horizon of a slow variable that is coupled to solution
  components that act on a fast scale.
  Although the fast scale variables are essential for the dynamics of the coupled problem, 	  they are often of no interest in themselves.
  Recently we have proposed a temporal multiscale
  approach that fits into the framework of the heterogeneous multiscale
  method and that allows for efficient simulations with significant
  speedups. Fast and slow scales are decoupled by introducing local
  averages and by replacing fast scale contributions by localized
  periodic-in-time problems.
  Here, we generalize this multiscale approach to a
  larger class of problems but in particular, we derive an a
  posteriori error estimator based on the dual weighted residual
  method that allows for a splitting of the error into averaging
  error, error on the slow scale and error on the fast scale. We
  demonstrate the accuracy of the error estimator and also its use
  for adaptive control of a numerical  multiscale scheme. 
\end{abstract}

\section{Introduction}
We are interested in the efficient approximation of dynamical systems with multiple scales in time. Such problems appear in various applications such as material damage mechanics~\cite{Murakami2012}, astrophysics~\cite{BehounkovaTobieChobletCadek2010} or cardiovascular settings~\cite{FreiRichterWick2016,MizerskiRichter2020}. Although multiscale problems are extensively studied in literature, see e.g.\,\cite{Cioranescu, Oleinik}, most works focus on problems where the multiscale character is in space but not in time.   The \emph{heterogeneous multiscale method (HMM)}~\cite{EEngquist2003,E2011,Abdulleetal2012,EngquistTsai2005} is a very general framework and easily applied to temporal multiscale dynamics. Here, fast and slow problems are decoupled by means of an averaging that gives an effective equation for the slow dynamics. The feedback between both scales is realized by localized fine scale problems that have to be approximated in every time step of the slow problem.

In~\cite{FreiRichter2020,MizerskiRichter2020} we have developed such a multiscale approach with applications to medical flow problems, where the slow scale describes the growth of a stenosis and where the fast problem is the oscillatory dynamics coming from heart driven blood flow. For decoupling the scales local periodic-in-time solutions describing the fast dynamics are introduced and have to be solved once in each time step of the slow problem. An a priori error estimate for this multiscale scheme has been shown for a simple problem based on the Stokes equation. Numerically, significant acceleration and a reduction of the computational time by a factor of up to $10\,000$ is observed as compared to fully resolved simulations.

Algorithmically such multiscale schemes are complex, as they are based on multiple discretization schemes for slow and fast scales and as they require careful control of the transmission operator that carries information between these scales. Time step sizes must be chosen for the slow and the fast scale and further, tolerances must be defined to control the approximation quality all problems that are involved in the multiscale approach. If one or both of the scales are described by partial differential equations it is also necessary to control the spatial discretization parameters.

Here, we will derive and discuss an a posteriori error estimator based on the dual weighted residual method~\cite{BeckerRannacher2001} for estimating functional errors that cover all error contributions coming from discretization and multiscale approximation. By splitting and localizing the error estimator to the various components, an adaptive multiscale scheme is realized that allows to optimally balance the different error contributions. The concept of goal oriented error estimation is chosen, since it allows for a uniform handling of temporal~\cite{SchmichVexler2008} and spatial discretization errors~\cite{BeckerRannacher1995} but also of truncation errors coming from the violation of conformity by the non-exact solution of sub problems~\cite{MeidnerRannacherVihharev2009} and finally it also allows to include the multiscale error which can be considered as a kind of model error~\cite{BraackErn2003}. Since the design of the error estimator is complex and involves a staggered approach with adjoint and tangent problems on both scales, we restrict the presentation to a system of ordinary differential equations to keep the notation and discussion as brief as possible. 

In the next section we describe the problem under consideration and we briefly summarize the multiscale approximation scheme as introduced in~\cite{FreiRichter2020}. Then, Section~\ref{sec:disc} casts the multiscale scheme into a temporal Galerkin formulation that will act as basis for the error estimator derived in Section~\ref{sec:error}. Numerical examples are discussed in Section~\ref{sec:num} and finally, we summarize in a short conclusion.

\section{Model problem and multiscale approximation}\label{sec:modelproblem}
We consider a system of ordinary differential equations.

\begin{problem}[Model Problem]\label{problem:model}
  On $I=[0,T]$ find $y:I\to\mathds{R}^c$ and $u:I\to\mathds{R}^d$ with $c,d\in\mathds{N}$ such that
  \begin{equation}\label{modelproblem}
    y'(t)= \epsilon f\big(y(t),u(t)\big),\quad 
    u'(t) = g\big(t,y(t),u(t)\big)
  \end{equation}
  with $y(0)=y_0\in\mathds{R}^c$ and $u(0)=u_0\in\mathds{R}^d$ and 
  the scale separation parameter $0<\epsilon\ll 1$. We will call $y(t)$ the slow component and $u(t)$ the fast component of the problem.
\end{problem}
Let the following assumptions hold.
\begin{assumption}[Slow Scale]
  \label{assumption:f}
  Let $f$ be continuous on $I\times \mathds{R}^c\times \mathds{R}^d$, 
  bounded
  \begin{equation}\label{ass:1Bf}
    \| f(y,u) \|
    \le \Cas{1Bf}\qquad \forall (y,u)\in  \mathds{R}^c\times \mathds{R}^d, 
  \end{equation}
  and Lipschitz and differentiable with respect to $y$ and $u$, such that 
  \begin{equation}\label{ass:1Lf}
    \begin{aligned}
      \|f(y,u)-f(Y,u)\| &\le \Cas{1Lf} \|y-Y\|\quad\forall
      y,Y\in\mathds{R}^c,\quad \forall u\in \mathds{R}^d\\
      \|f(y,u)-f(y,U)\| &\le \Cas{1Lf} \|u-U\|\quad\forall
      u,U\in\mathds{R}^d,\quad \forall y\in  \mathds{R}^c.
    \end{aligned}
  \end{equation}
\end{assumption}
\begin{remark}[Separation of temporal scales]
  By the boundedness of $f$, it holds
  \begin{equation}\label{ass:boundc}
    \|y'(t)\| \le \Cas{1Bf}\cdot\epsilon,\quad \|y(t)\| \le
    \Cas{boundc}\text{ for } t\in [0,T],
  \end{equation}
  with $\Cas{boundc}=\Cas{1Bf}\cdot \epsilon\cdot T + \|y_0\|$.
  We define
  \[
  X^c:= \{ y\in\mathds{R}^c,\; \|y\|\le \Cas{boundc}\}. 
  \]
\end{remark}
On the fast scale problem we impose the following assumptions.
\begin{assumption}[Fast Scale]\label{assumption:g}
  Let $g$ be continuous on $I\times \mathds{R}^c\times \mathds{R}^{d}$ 
  and Lipschitz 
  \begin{equation}\label{ass:1Lg}
    \|g(t,y_1,u)-g(t,y_2,u)\|\le \Cas{1Lg}\|y_1-y_2\|
    \text{ for all } y_1,y_2\in X^c,
  \end{equation}
  uniform in $t\in I$ and $u\in \mathds{R}^d$. We assume that $g$ is
  differentiable with respect to $y\in\mathds{R}^c$ and
  $u\in\mathds{R}^d$ with a Jacobian $\nabla_ug$ whose eigenvalues all
  have negative real part 
  \begin{equation}\label{ass:negdef}
    -\Las{negdef}:=\sup\{Re(\lambda),\; \lambda\text{ is Eigenvalue of
    }\nabla_u 
    g\big(t,y,u\big)\} <0,
  \end{equation}
  uniform in $t\in [0,T]$ and $y\in X^c$.   
  Finally, let $g$ be
  $1$-periodic in time
  \begin{equation}\label{ass:1Pg}
    g(t,y,u) = g(t+1,y,u)\quad\forall (t,y,u)\in \mathds{R}\times
    \mathds{R}^c\times \mathds{R}^d
  \end{equation}
  and we assume that for each $y\in X^c$ there exists
  a unique periodic-in-time solution to
  \begin{equation}\label{ass:perproblem}
    u'_y(t) = g\big(t,y,u_y(t)\big)\quad t\in I^P:=[0,1],\quad
    u_y(0)=u_y(1),
  \end{equation}
  which is bound by
  \begin{equation}\label{ass:boundu}
    \max_{t\in I^P}\|u_y\|\le \Cas{boundu},
  \end{equation}
  with a constant $\Cas{boundu}=C(\Cas{boundc})$. Finally, we assume
  that for all $y_1,y_2\in X^c$ and the corresponding periodic-in-time 
  solutions $u_{y_1},u_{y_2}:[0,1]\to \mathds{R}^d$ it holds
  \begin{equation}\label{ass:uclip}
    \sup_{t\in I^P} \|u_{y_1}(t)-u_{y_2}(t)\| \le \Cas{uclip}
    \|y_1-y_2\|. 
  \end{equation}
\end{assumption}

\begin{remark}[Scope of applications]\label{rem:applications}
In Section~\ref{sec:num} we will discuss specific systems of
differential equations that meet this assumptions.  The examples
considered in \cite{FreiRichter2020} and \cite{MizerskiRichter2020}
also fit into the frame of assumptions given above. In particular for
nonlinear equations, like the Navier-Stokes equations which are
considered in~\cite{MizerskiRichter2020}, the existence of
time-periodic solutions can only be shown under very strict
assumptions on the problem data, with limitations to small Reynolds
numbers~\cite{GaldiKyed2016a}. Hence, in
particular~(\ref{ass:perproblem})-(\ref{ass:uclip}) will be difficult
to validate in most application cases. Also,~(\ref{ass:negdef}) calls
for a damping behavior of the fast scale problem and is here used to
obtain uniqueness of solutions. Considering the Navier-Stokes
equations this also calls for bounds on the problem data. In
Section~\ref{sec:num} we will present a simple test case where all
assumptions can be validated.
\end{remark}

The multiscale scheme for the efficient approximation of
Problem~\ref{problem:model} is based on~\cite{FreiRichter2020} and on
the following averaged multiscale problem.
\begin{problem}[Averaged Multiscale Problem]\label{problem:multiscale}
  On $I=[0,T]$ find $Y: I\to \mathds{R}^c$ such that
  \begin{equation}\label{prob:1}
    Y'(t) = \epsilon \int_t^{t+1}
    f\big(Y(t),u_{Y(t)}(s)\big)\,\text{d}s,\quad Y(0)=y_0,
  \end{equation}
  and where $u_{Y(t)}(\cdot)$, for each $Y(t)\in\mathds{R}^c$, is given
  as the solution to the periodic-in-time micro problem
  \begin{equation}\label{prob:2}
    \frac{d}{ds} u_{Y(t)}(s) = g\big(s,Y(t),u_{Y(t)}(s)\big)\text{ in
    }I^P=[0,1],\quad
    u_{Y(t)}(1) = u_{Y(t)}(0). 
  \end{equation}
\end{problem}
Defining a \emph{fast scale feedback operator} ${\cal
  F}:\mathds{R}^c\to \mathds{R}^c$ by
\begin{equation}\label{prob:3}
  {\cal F}(Y):=\int_t^{t+1}f\big(Y(t),u_{Y(t)}(s)\big)\,\text{d}s,
\end{equation}
with $u_{Y(t)}$ given by~(\ref{prob:2}), the averaged multiscale
problem can be written as
\begin{equation}\label{prob:4}
  Y'(t) = \epsilon {\cal F}\big(Y(t)\big),\quad Y(0)=y_0,
\end{equation}
and the fast scale problem is formally removed. This problem notation
is basis for the numerical multiscale method to be described in
Section~\ref{sec:disc}. We will discretize~(\ref{prob:4}) with very
large time steps, and, in each time step $Y_{n-1}\mapsto Y_n$, the fast
scale feedback operator ${\cal F}(\cdot)$ must be evaluated in the approximations $Y_{n-1}$ and $Y_n$.
The clue in the design of this multiscale approach is the introduction
of locally periodic-in-time micro solutions $u_{Y(t)}(\cdot)$.  This
allows to formally decouple the microscale from the macroscale. ${\cal
  F}(Y_n)$ can be approximated without requiring any initial values that
might have to be obtained from the previous macro step solution
$Y_{n-1}$.

In the following we will show that this averaged multiscale problem
has a solution and we will show that this solution $Y(t)$ is close to
the original solution $y(t)$.

\subsection{Analysis of the multiscale error}

Given the above listed assumptions on the slow and the fast scale, we
can show that the solution to the multiscale problem,
Problem~\ref{problem:multiscale} is close to the resolved original
solution. The following theorem is a generalization of~\cite[Lemma
  10]{FreiRichter2020} to a more general class of equations.
\begin{theorem}[Multiscale error]\label{thm:error}
  Let Assumptions~\ref{assumption:f} and~\ref{assumption:g} hold. There exists a unique solution $Y(t)$ to Problem~\ref{problem:multiscale}. Let $\big(y(t),u(t)\big)$ be the solution to Problem~\ref{problem:model}. It holds
  \[
  \max_{t\in [0,T]} \|y(t)-Y(t)\| = {\cal O}(\epsilon). 
  \]
\end{theorem}
\begin{proof}
  Let $Y_1,Y_2:I\to X^c$ be given. For the right hand side of Problem~\ref{problem:multiscale}, given in the form~(\ref{prob:4}) it holds
  \[
  \begin{aligned}
    \|{\cal F}(Y_1(t))-{\cal F}(Y_2(t))\|
    &\le \int_t^{t+1}
    \|f\big(Y_1(t),u_{Y_1(t)}(s)\big)-
    f\big(Y_2(t),u_{Y_2(t)}(s)\big)\|\,\text{d}s\\
    &\le \Cas{1Lf}\int_t^{t+1}
    \|Y_1(t)-Y_2(t)\|+
    \|u_{Y_1(t)}(s)-u_{Y_2(t)}(s)\|
    \,\text{d}s\\
    &\le \Cas{1Lf}(1+\Cas{uclip})\|Y_1(t)-Y_2(t)\|. 
  \end{aligned}
  \]
  Hence, the problem is Lipschitz and a unique solution exists. The bound $\|Y(t)\|\le \Cas{1Bf}\cdot\epsilon\cdot t$ shows that this solution exists on all $I=[0,T]$. 
  
  To start with, we introduce the average
  \[
  \bar Y(t):=\int_t^{t+1}y(s)\,\text{d}s.
  \]
  For this average it holds
  \begin{equation}\label{thm:A}
    \|y(t)-\bar Y(t)\|
    \le \int_t^{t+1}\|y(t)-y(s)\|\,\text{d}s
    =\int_t^{t+1}\|\int_t^sy'(r)\,\text{d}r\|\,\text{d}s
    \le \frac{\Cas{1Bf}}{2}\epsilon.
  \end{equation}
  The remaining error $w(t):=\bar Y(t)-Y(t)$ is governed by
  \begin{equation}\label{thm:0}
    w'(t)
    = \epsilon \int_t^{t+1}
    f\big(y(s),u(s)\big)
    - f\big(Y(t),u_{Y(t)}(s)\big)
    \,\text{d}s,\quad
    w(0) = \int_0^1 y(s)\,\text{d}s-y_0,
  \end{equation}
  where the initial value is estimated with help of (\ref{thm:A})
  \begin{equation}\label{thm:1}
    \|w(0)\| \le \int_0^1 \|y(s)-y(0)\|\,\text{d}s \le
    \frac{\Cas{1Bf}}{2}\epsilon. 
  \end{equation}
  To estimate $w(t)$ we split the right hand side of~(\ref{thm:0})
  into 
  \begin{multline}\label{thm:2}
    f\big(y(s),u(s)\big)
    - f\big(Y(t),u_{Y(t)}(s)\big)\\
    =
    \Big(
    f\big(y(s),u(s)\big)
    - f\big(Y(t),u(s)\big)
    \Big)
    +
    \Big(
    f\big(Y(t),u(s)\big)
    - f\big(Y(t),u_{Y(t)}(s)\big)
    \Big).
  \end{multline}
  Then, Lipschitz continuity of $f(\cdot)$ gives
  \begin{equation}\label{thm:3}
    \|w'(t)\| \le \Cas{1Lf} \epsilon\int_t^{t+1}
    \Big( \|y(s)-Y(t) \| + \|u_{Y(t)}(s)-u(s)\|\Big)\,\text{d}s.     
  \end{equation}
  The first term is bounded by introducing $\pm \bar Y(t)$
  \begin{equation}\label{thm:4}
    \begin{aligned}
      \int_t^{t+1}\|y(s)&-Y(t) \|\,\text{d}s
      \le \int_t^{t+1}\|y(s)-\bar Y(t)\|\,\text{d}s + \|\bar Y(t)-Y(t)\|\\
      & \le \int_t^{t+1}\int_t^{t+1}\|y(s)-y(r)\|\,\text{d}r\,\text{d}s
      + \|w(t)\|\\ 
      &= \int_t^{t+1}\int_t^{t+1}\|\int_r^s
      y'(t)\,\text{d}t\|\,\text{d}r\,\text{d}s  +  \|w(t)\|
      \le \frac{\Cas{1Bf}}{3}\epsilon + \|w(t)\|
    \end{aligned}
  \end{equation}
  To bound the second term in~(\ref{thm:3}) we introduce $\pm
  u_{y(s)}(s)$
  \begin{equation}\label{thm:5}
    \int_t^{t+1} \|u_{Y(t)}(s)-u(s)\|\,\text{d}s
    \le 
    \int_t^{t+1}
    \|u_{Y(t)}(s)-u_{y(s)}(s)\|+
    \|u_{y(s)}(s)-u(s)\|\,\text{d}s
  \end{equation}
  Here, the first term is bounded with help of~(\ref{ass:uclip})
  and~(\ref{thm:A})
  \begin{equation}\label{thm:6}
    \begin{aligned}
      \int_t^{t+1}
      \|&u_{Y(t)}(s)-u_{y(s)}(s)\|\,\text{d}s
      \le \Cas{uclip}\int_t^{t+1} \|Y(t)-y(s)\|\,\text{d}s\\
      &\le \Cas{uclip}\Big( \|Y(t)-\bar Y(t)\|
      +  \|\bar Y(t)-y(t)\|
      + \int_t^{t+1} \| y(t)-y(s)\|\,\text{d}s
      \Big)\\
      &\le \Cas{uclip}\Big( \|w(t)\| + \Cas{1Bf}\epsilon\Big).
    \end{aligned}
  \end{equation}

  The second term in~(\ref{thm:5}), $\|u_{y(s)}(s)-u(s)\|$, is more
  subtle to estimate. It measures the difference between the fully
  dynamic solution $u(s)$, evolving around the slow scale $y(s)$ to
  the isolated periodic-in-time solutions for each $y(s)$ fixed. For
  better readability we move the estimate of this term to the separate
  Lemma~\ref{lemma:difffast} which shows
  \begin{equation}\label{thm:7}
    \int_t^{t+1} \|u_{y(s)}(s)-u(s) \|\,\text{d}s
    \le \Cas{diffast} \cdot \epsilon.
  \end{equation}
  Combining this with~(\ref{thm:1})-(\ref{thm:6}) gives
  \begin{equation}\label{thm:8}
    \|w'(t)\| \le \Cas{1Lf}\epsilon \Big(
    \big(1+\Cas{uclip}\big)\|w(t)\| +
    \big(
    \frac{\Cas{1Bf}}{3}
    + \Cas{uclip} + \Cas{diffast}
    \big)
    \epsilon\Big),
  \end{equation}
  where $\|w(0)\| \le \frac{\Cas{1Bf}}{2}\epsilon$. Using~\cite[Sec. 3]{Conti1956} we can estimate $\|w(t)\|$ by the solution to the differential equation with the corresponding right hand side
  \begin{equation}\label{thm:9}
    \|w(t)\|
    \le \Big(1+ \exp\big(
    \Cas{1Lf}(1+\Cas{uclip})\epsilon\cdot t
    \big)\Big)
    \big(
    \Cas{1Bf}+\Cas{uclip}+\Cas{diffast}
    \big)\epsilon,
  \end{equation}
  which for $t = {\cal O}(\epsilon^{-1})$ is bound by
  \begin{equation}\label{ass:thmC}
    \|\bar Y(t)-Y(t)\| \le \Cas{thmC} \epsilon,
  \end{equation}
  with a constant
  $\Cas{thmC}=C(\Cas{1Bf},\Cas{1Lf},\Cas{uclip},\Cas{diffast})$.  Finally, the claim of the theorem follows by   combining~(\ref{thm:A}) and~(\ref{ass:thmC})
  \[
  \|y(t)-Y(t)\| \le \|y(t)-\bar Y(t)\| + \|\bar Y(t)-Y(t)\|
  \le \frac{\Cas{1Lf}}{2}\epsilon + \Cas{thmC} \epsilon. 
  \]
\end{proof}

It remains to prove estimate~(\ref{thm:7}), the difference between the 
dynamically evolving fast scale solution and the local
periodic-in-time solutions. This lemma is close
to~\cite[Lemma 9]{FreiRichter2020}. 
\begin{lemma}\label{lemma:difffast}
  Let $y\in C^1(I;X^c)$ be given such that
  $y(0)=y_0$ and $\|y'(t)\| \le \Cas{1Bf}\cdot \epsilon$.
  Let $u_{y(t)}(s)$ be the family of periodic-in-time fast scale
  solutions for $y(t)$ fixed and let $u(t)$ be the dynamic solution to
  \[
  u'(t) = g\big(t,y(t),u(t)\big),\quad 
  u(0) = u_{y(0)}(0). 
  \]
  There exists a constant
  $\Cas{diffast}:=C(\Cas{1Bf},\Cas{uclip},\Las{negdef})$ 
  such that 
  \begin{equation}\label{ass:diffast}
    \| u(t)-u_{y(t)}(t)\| \le \Cas{diffast}\cdot \epsilon.
  \end{equation}
\end{lemma}
\begin{proof}
  For the periodic-in-time solutions it holds
  \[
  \frac{d}{dt} u_{y(t)}(t) = u'_{y(t)}(t)
  + \underbrace{\frac{du_{y(t)}}{dy(t)}(t)}_{=:Du_{y(t)}(t)} y'(t),
  \]
  and the difference $v(t):=u(t)-u_{y(t)}(t)$ is governed by
  \[
  v'(t)  = g\big(t,y(t),u(t)\big)
  - g\big(t,y(t),u_{y(t)}(t)\big) - Du_{y(t)}(t)y'(t),
  \]
  and, since $g$ is differentiable, it holds
  \[
  v'(t)  = \nabla_y g\big(t,y(t),\gamma(t)\big)v(t) -
  Du_{y(t)}(t)y'(t),
  \]
  where $\gamma(t)\in \mathds{R}^u$ is an intermediate between $u(t)$ and $u_{y(t)}(t)$. 
  Since $v(0)=u_{y(0)}(0)-u_{y(0)}(0)=0$  it holds
  \[
  v(t) = \int_0^t \Phi(s) Du_{y(s)}(s)y'(s)\,\text{d}s,
  \quad
  \Phi(t) = \exp\Big(\nabla_y g\big(t,y(t),\gamma(t)\big)\cdot t\Big),
  \]
  such that
  \begin{equation}\label{der1}
    \|v(t)\| \le 
    \Cas{1Bf} \epsilon \int_0^t\Phi(s)\,\text{d}s
    \cdot \sup_{s\in [0,t]} | Du_{y(s)}(s)|
    \le \sup_{s\in [0,t]} | Du_{y(s)}(s)| \cdot
    \frac{\Cas{1Bf}}{\Las{negdef}}\epsilon,
  \end{equation}
  Using that all eigenvalues of $\nabla_y g$  have a negative real
  part. 
  To estimate $Du_y(t)$ we first note that $g$ is differentiable
  with respect to $y\in\mathds{R}^c$ such that
  \[
  Du_y'(t) =\nabla_u g\big(t,y,u_y(t)\big) Du_y(t) +  \nabla_y
  g\big(t,y,u_y(t)\big). 
  \]
  This equation is linear in $Du_y(t)$ such that a unique solution
  exists. Given two periodic solutions $u_{Y_1}$ and $u_{Y_2}$ and
  using~(\ref{ass:uclip}) we can bound the derivative
  $Du_y(t)$ by estimating
  \begin{equation}\label{der2}
    \frac{ |u_{Y_1}(t)-u_{Y_2}(t)|}{|Y_1-Y_2|} \le
    \Cas{uclip},
  \end{equation}
  uniform in $Y_1,Y_2\in\mathds{R}^d$. Hence, combining~(\ref{der1})
  and~(\ref{der2}) 
  \begin{equation}\label{thm:10}
    \|u(t)-u_{y(t)}(t)\|\le \frac{\Cas{1Bf}\Cas{uclip}}{\Las{negdef}}
    \epsilon. 
  \end{equation}
\end{proof}

Theorem~\ref{thm:error} is the analytical basis for the following
numerical approximation scheme. We have shown that the solution to the
averaged multiscale problem, Problem~\ref{problem:multiscale} is close
to the solution of the  original model problem. The numerical scheme
will be based on the approximation of the averaged equation in Problem~\ref{problem:multiscale} using large time steps $K$,
which are by far larger than the micro scale. For evaluating the right
hand side, we will have to evaluate the transfer operator ${\cal
  F}\big(Y(t)\big)$. This will require the solution of a
periodic-in-time micro problem. Here a small time step $k\ll K$
must be employed. In~\cite{FreiRichter2020} we have further given an a
priori error estimator for the numerical discretization of the
averaged multiscale method and, in~\cite[Theorem 18]{FreiRichter2020} we
have shown that the multiscale scheme converges with second order, 
${\cal O}(\epsilon^2 K^2 + k^2)$, if both averaged macro scale problem and
periodic micro problem are discretized with second order time stepping
schemes.

\subsection{Variational formulation}

We conclude by defining the variational formulation of the averaged 
multiscale problem which will be the basis for the Galerkin
discretization and the a posteriori error estimator. 

\begin{problem}[Variational formulation of the multiscale problem]
  \label{primal:problem} 
  Find $Y\in \X$ such that 
  \begin{equation}\label{var:long}
    \begin{aligned}
      A(Y,\Phi) &= 0\quad \forall \Phi\in \Y,\\
      A(Y,\Phi) \coloneqq
      \int_0^T \big( Y'(t)-\epsilon{\cal F}\big(Y(t)\big)\big)&\cdot
      \Phi(t)\,\text{d}t,\quad 
          {\cal F}(Y) \coloneqq \int_0^1 f\big(Y,u_Y(s)\big)\,\text{d}s
    \end{aligned}
  \end{equation}
  where
  \begin{equation}\label{spaces:long}
    \begin{aligned}
      \X\coloneqq \{ \Phi\in H^1(I;\mathds{R}^c),\; \Phi(0) = y_0 \},\quad
      \Y\coloneqq L^2(I;\mathds{R}^c),  
    \end{aligned}
  \end{equation}
  and where, for a fixed value $Y\in\mathds{R}^c$, the periodic fast
  scale solutions $u_Y$ are defined  on $I^P=[0,1]$ by 
  \begin{equation}\label{var:fast}
    \begin{aligned}
      u_Y\in \V^\pi:\quad 
      B(Y;u_Y;\phi) &:
      \int_0^1 \big( u'_Y(t)- g\big(t,Y,u_Y(t)\big)\phi(t)\,\text{d}t
      =0\quad \forall \phi\in \W,
    \end{aligned}
  \end{equation}
  with test and trial spaces defined as
  \begin{equation}\label{spaces:fast}
    \V^\pi\coloneqq \{\phi\in H^1(I^P;\mathds{R}^d),\; \phi(0)=\phi(1)\},\quad
    \W \coloneqq L^2(I^P;\mathds{R}^d). 
  \end{equation}
\end{problem}

\section{Discretization and Solution}
\label{sec:disc}

Discretization of~(\ref{prob:4}) is based on a temporal Galerkin scheme of Problem~\ref{primal:problem}. For general literature on temporal Galerkin formulations we refer to~\cite{Thomee1997,ErikssonEstepHansboJohnson1995}. Both long term and short term problem are discretized with continuous and piecewise linear functions using piecewise constant test functions with possible discontinuities at the discrete time steps. This results in a time stepping scheme which is of second order and which is, up to numerical quadrature, equivalent to the trapezoidal rule.
We introduce the partitioning $I_K$ of $I=[0,T]$ by
\begin{equation}\label{partitioningK}
  0=T_0<T_1<\cdots<T_N = T,\quad K_n:=T_n-T_{n-1},\quad
  I_n\coloneqq (T_{n-1},T_n]
\end{equation}
and define the discrete subspaces 
\begin{equation}\label{spaces:slow:disc}
  \begin{aligned}
    \X_K&\coloneqq \{\Psi\in C(\bar I)^c\,:\, \Psi\big|_{I_n}\in
    P^1(I_n;\mathds{R}^c),\; n=1,\dots,N,\; \Psi(0)=y_0\}\subset \X\\
    \Y_K&\coloneqq \{\Phi\in L^2(I;\mathds{R}^c)\,:\, \Phi\big|_{I_n}\in P^0(I_n;\mathds{R}^c),
    \; n=1,\dots,N\}
    \subset \Y,
  \end{aligned}
\end{equation}
where we denote by $P^r(I)=\operatorname{span}\{1,t,\dots,t^r\}$ the
space of polynomials up to degree $r\in\mathds{N}$. 
Likewise, for discretization of the micro problems~(\ref{var:fast}) we introduce
partitionings $I^P_n$ of $I^P=[0,1]$ by defining
\begin{equation}\label{partitioningk}
  0=t_n^0<t_n^1<\cdots < t_n^{M_n}=1,\quad
  k_n\coloneqq t_n^m-t_n^{m-1},\quad I^P_{n,m}\coloneqq (t_n^{m-1},t_n^m].
\end{equation}
While we allow for different micro discretizations in each
macro step $n=1,\dots,N$, we assume that each of them is uniform with step size $k_n$. 
We introduce
\begin{equation}\label{spaces:fast:disc}
  \begin{aligned} 
    \V_{k;n}^\pi&\coloneqq \{\phi\in C(I^P;\mathds{R}^d):\phi\big|_{I^P_{n,m}}\in P^1(I^P_{n,m})^d,\;
    m=1,\dots,M_n,\; \phi(1)=\phi(0) \}\subset \V^\pi,\\
    \W_{k;n}&\coloneqq \{\phi\in L^2(I^P;\mathds{R}^d):\phi\big|_{I^P_{n,m}}\in P^0(I^P_{n,m})^d,\;
    m=1,\dots,M_n\} \subset \W. 
  \end{aligned}
\end{equation}
Mostly, we will skip the index $n$ if we refer to these micro spaces. 
Discretization is accomplished by restricting trial and test functions to the discrete function spaces $\X_K,\Y_K$ and $\V_{k;n}^\pi,\W_{k;n}$, respectively.

For general right hand sides $f(\cdot)$ and $g(\cdot)$, the integrals
appearing in the variational formulations~(\ref{var:long})
and~(\ref{var:fast}) cannot be evaluated exactly, since the trial
spaces $\X_K$ and $\V_{k;n}^\pi$ are piecewise linear and these
functions are potentially nonlinear. Instead, we numerically approximate them with a summed two-point Gaussian quadrature rule, which we define by
\begin{equation}\label{quadrature}
  \int_{I_K} f(t)\,\mathrm{d}t \coloneqq
  \sum_{n=1}^N \frac{K_n}{2}
  \Big(
  f\big(\bar T_n-\frac{K_n}{\sqrt{12}}\big)+
  f\big(\bar T_n+\frac{K_n}{\sqrt{12}}\big)
  \Big),\quad \bar T_n\coloneqq \frac{T_{n-1}+T_n}{2}. 
\end{equation}
Integration on each micro partitioning $I^P_n$ is defined in the same spirit.
This quadrature rule is of fourth order, see~\cite{StoerBulirsch2002},
and it guarantees the additional higher order consistency error ${\cal O}(\epsilon K^4+k_n^4)$ in  contrast to all other error terms which are of order two. Hence from here on we will neglect the conformity error coming from numerical quadrature on both scales. 

Altogether, the fully discrete multiscale solution is described by the following problem formulation:
\begin{problem}[Discretized variational formulation of the multiscale
    problem]\hfill
  
  \label{primaldisc:problem} Find $Y_{K,k}\in \X_K$ such that 
  \begin{equation}\label{vardisc:long}
    \begin{aligned}
      A_{k}(Y_{K,k},\Phi) &= 0 \quad \forall \Phi\in \Y_K,\\
      A_k(Y,\Phi) \coloneqq
      \int_0^T \big( Y'(t)-\epsilon{\cal F}_k\big(Y(t)\big)\big)&\cdot
      \Phi(t)\,\mathrm{d}t,\quad 
          {\cal F}_k(Y) \coloneqq \int_0^1 f\big(Y,u_{k;Y}(s)\big)\,\mathrm{d}s
    \end{aligned}
  \end{equation}
  where $\X_K,\Y_K$ are given in~(\ref{spaces:slow:disc}). 
  For $t\in I_n$ and $Y=Y(t)\in\mathds{R}^c$ fixed, the discrete periodic fast scale solution $u_{k;Y}\in \V_{k;n}^\pi$ is defined by 
  \begin{equation}\label{vardisc:fast}
    u_{k;Y}\in \V^\pi_{k;n}:\;
    B(Y;u_{k;Y};\phi_k) = 0\quad \forall \phi_k\in \W_{k;n},
  \end{equation}
  where the function spaces $\V^\pi_{k;n}$ and $\W_{k;n}$ are given in~(\ref{spaces:fast:disc}) and $B(\cdot)$ in~(\ref{var:fast}). 
\end{problem}

\begin{remark}[Efficiency of Galerkin discretizations]
  The approximation of Galerkin time discretizations with high order quadrature rules (the approach that we describe) causes additional effort since at least two evaluations of all nonlinear operators and functions are required in each step. The closely related trapezoidal rule would only require one evaluation. However, the consistency error between the Galerkin approach and the trapezoidal rule is of the same order such that both approaches must be considered as separate discretization schemes. In~\cite{MeidnerRichter2014,MeidnerRichter2015} we have demonstrated how the more efficient trapezoidal rule can be used for solving the problem while including the consistency error within the error estimator. It shows that  the quadrature error is indeed of the same (or even higher) order than further contributions to the error estimator. 
\end{remark}


We conclude by summarizing the algorithmic realization of the
multiscale process. The discrete variational formulation~(\ref{vardisc:long}) decouples into separate time steps on $I_n=[T_{n-1},T_n]$. On each interval it remains to solve a nonlinear problem for $Y_n\in P^1(I_n;\mathds{R}^c)$  
\begin{equation}\label{disc:step}
  Y_n(T_n)\cdot\Phi_n-\epsilon\int_{T_{n-1}}^{T_n}{\cal F}_k\big(Y_n(t)\big)\,\text{d}t\cdot \Phi_n = Y_{n-1}(T_{n-1})\cdot \Phi_n\quad \forall \Phi_n\in P^0(I_n;\mathds{R}^c). 
\end{equation}
Here, the integral is approximated by a two point Gaussian quadrature rule. In each quadrature point it is necessary to solve the periodic in time micro problem for $Y_{n,q}=Y_n(\chi_{n,q})$, $\chi_{n,q}\in I_n$ being, for $q=1,2$, the two Gauss points. The corresponding solutions $u_{Y_{n,q}}(t)$ are given by~(\ref{vardisc:fast}) and their approximation requires the solution of a periodic problem, see Remark~\ref{rem:periodic}.

To solve~(\ref{disc:step}), we use an approximate Newton scheme. We approximate the Jacobian by neglecting the derivatives of the transfer operator ${\cal F}$ with respect to the micro solution, i.e. we approximate
\[
\tilde \nabla_Y {\cal F}(Y)(\delta Y) := \int_0^1 \nabla_Yf\big( Y,u_Y(s)\big)\cdot \delta Y\,\text{d}s,
\]
and do not consider the partial derivative $\nabla_u f(Y,u_Y)\nabla_Y u_y$ which would require the approximation of a further periodic in time tangent problem. Numerical tests have shown that this approximation does not strongly worsen the convergence rate of the Newton scheme.

\begin{remark}[Approximation of periodic solutions]\label{rem:periodic}
  In each Gaussian quadrature point $\chi_{n,q}\in [T_{n-1},T_n]$,
  periodic-in-time micro problems $u_{k;Y_{n,q}}(t)$ must be computed
  for the fixed slow scale variable $Y_{n,q}:=Y_K(\chi_{n,q})$. These are approximated until the periodicity mismatch $\|u_{k;Y_{n,q}}(1)-u_{k;Y_{n,q}}(0)\|< tol_P$. This can either be done by simply letting the dynamic problem run into a cyclic state or by using different acceleration schemes, see~\cite{RichterWollner2018,Richter2020}. 
\end{remark}

\section{Error estimation}\label{sec:error}
We follow the framework of the dual weighted residual estimator (DWR) introduced in~\cite{BeckerRannacher1995,BeckerRannacher2001}. We are interested in functional outputs $J:\X\to\mathds{R}$ of the long scale problem.
We aim at estimating the functional error $J(y)-J(Y_{K,k})$ between the analytic solution $y(t)$ given by~(\ref{modelproblem})-(\ref{reaction}) and the fully discrete multiscale approximation defined in Problem~\ref{primaldisc:problem}. In between, we must consider several approximation steps:
\begin{enumerate}
\item The averaging error \EA\ introduced by deriving the averaged
  model problem, Problem~\ref{problem:multiscale}
  \[
  J(y)-J(Y_{K,k}) = \underbrace{\big(J(y) - J(Y)\big)}_{\EA} +
  \underbrace{\big(J(Y)-J(Y_{K,k})\big)}_{\ED}, 
  \]
  and the remaining discretization error \ED, which is further split.
\item The error from Galerkin discretization \EG\ of the averaged long term
  problem 
  \[
  J(y)-J(Y_{K,k}) = \underbrace{\big(J(y) - J(Y)\big)}_{\EA}
  + \underbrace{\big(J(Y)-J(Y_{K})\big)}_{\EG}
  + \underbrace{\big(J(Y_K)-J(Y_{K,k})\big)}_{\EF}
  \]
  which also reveals \EF, the error coming from discretizing the fast
  scale problem. 
\end{enumerate}
By $Y_K$ we define the solution to the  semidiscrete problem,
which is discrete in terms of the long scale, e.g. $Y_K\in \X_K$, but
which is based on the analytic transfer operator ${\cal F}$. This
intermediate solution will enter the estimate as an analytical tool
only.  

While the averaging error \EA\ is bound by the a priori estimate in Theorem~\ref{thm:error}, the remaining errors $\ED=\EG+\EF$ can be formulated as residual errors of a non conforming Galerkin formulation. Non conformity comes from the approximation of the transfer operator ${\cal F}$ by ${\cal F}_k$. As outlined above, we have neglected the error coming from Gaussian quadrature since it is negligible. 

The general framework of the dual weighted residual error estimator for such a non conforming discretization is discussed in \cite[Section 2.3]{BeckerRannacher2001} or \cite[Theorem 8.7]{Richter2017}.
An application to the multiscale scheme will require a nested application of the DWR method to also take care of the error coming from approximating the transfer operator ${\cal F}$ which implicitly depends on the fast scale contributions.
We state the main result.
\begin{theorem}[DWR estimator for the long term problem]
  \label{theorem:dwr}
  Let $I=[0,T]$ and let $y\in C^1(I;\mathds{R}^c)$ be the solution to~(\ref{modelproblem})-(\ref{reaction}) and $Y_{K,k}\in \X_K$ be the fully discrete solution to  Problem~\ref{primaldisc:problem}. Let $tol_P>0$ be the tolerance for the approximation of the temporal periodicity in all micro problems, i.e. $\|u_Y(1)-u_Y(0)\|<tol_P$. Let $J:\X\to\mathds{R}$ be three times differentiable. It holds
  \begin{multline}\label{errorestimate}
    J(y)-J(Y_{K,k}) = 
    -\frac{1}{2} A_k(Y_{K,k},Z-i_\Y Z) \\
    +\frac{1}{2}\Big(
    J'(Y_{K,k})(Y-i_\X Y) - A'_k(Y_{K,k})(Z_{K,k},Y-i_\X Y) \Big)\\
    +\frac{1}{2}\epsilon\int_0^T \eta^\pi\big(Y_{K,k}(t)\big)\cdot
    \big( Z(t)+Z_{K,k}(t)\big)\,\text{d}s\\
    +{\cal O}(tol_P) + {\cal O}(\epsilon)
    +{\cal O}(\epsilon k^2 + k^4 + K^4)
    +{\cal R}_K^{(3)} + {\cal R}_k^{(3)},
  \end{multline}
  where $i_\X:\X\to \X_K$ is the nodal interpolation into the space of
  piecewise linear polynomials and $i_\Y:\Y\to \Y_K$ is the projection
  to the piecewise constants, $Z\in \Y$ and $Z_{K,k}\in \Y_K$ are the adjoint solutions to
  $A'(Y)(\Phi,Z) = J'(Y)(\Phi)$ for all $\Phi\in \X^0$ and $\Phi_K\in
  \X_K^0$, respectively. $\X^0$ and $\X^0_K$ differ from $\X$ and
  $\X_K$ in the sense that homogeneous initial values are realized. 
  The fast scale error $\eta^\pi(Y_{K,k})$ is given by
  \begin{multline}\label{micro:eta}
    \eta^\pi(Y)\coloneqq   {\cal O}(tol_P) 
    +\frac{1}{2}\Big( G(z_Y-i_\W z_Y) - B(u_{k;Y}, z_Y-i_\W z_Y) \Big)\\
    + \frac{1}{2}\Big(J^{\pi'}(u_{k;Y})(u-i_\V u)-
    B'(u_{k;Y})(u-i_\V u, z_{k;Y})
    \Big)
  \end{multline}
  and the adjoint micro scale solutions 
  $z_Y\in \W$ and
  $z_{k;Y}\in \W_k$ are defined for each fixed $Y$ by
  $B'(u_Y)(\phi,z_Y)=\int_0^1\nabla_u f\big(Y,u_Y(s)\big)\phi(s)\,\text{d}s$ for all $\phi\in \V^\pi$ and $\phi_k\in \V_k^\pi$, respectively.
  $i_\V:\V^\pi\to \V_k^\pi$ and $i_\W:\W\to \W_k$ are interpolation operators.  By  ${\cal R}_K^{(3)}$ and ${\cal R}_k^{(3)}$ we denote remainders which are of third order in the error. 
\end{theorem}
\begin{proof}
  The proof follows by combining Theorem~\ref{thm:error}, 
  Lemma~\ref{dwr:long},~\ref{lemma:primal} and Remark~\ref{remark:adjointconsistency}. Details on the adjoint problems are given in Section~\ref{sec:evaluation}. 
\end{proof}


\subsection{Derivation of the error estimator}

The averaging error \EA\ $J(y)-J(Y)$ is estimated by a priori arguments. Given a differentiable functional $J(y)$ it holds with Theorem~\ref{thm:error} that
\[
|\EA| = |J(y)-J(Y)| = |J'(\zeta)(y-Y)|
\le \|J'(\zeta)\|\cdot \|y-Y\|
= {\cal O}(\epsilon),
\]
where $\zeta$ is an intermediate between $y$ and $Y$. 
We turn our attention to the Galerkin error \EG\ estimating $J(Y)-J(Y_{K,k})$. 
\begin{lemma}[DWR estimator of the averaged long term problem]
  \label{dwr:long}\

  \noindent
  Let $Y\in \X$ be the solution to Problem~\ref{primal:problem} and
  $Y_{K,k}\in \X_K$ be the solution to
  Problem~(\ref{primaldisc:problem}), $Z\in \Y$ and $Z_{K,k}\in \Y_K$ the adjoint solutions to~(\ref{de:1}). It holds
  \begin{multline}\label{error1}
    J(Y)-J(Y_{K,k}) = {\cal R}^{(3)}_K(Y-Y_{K,k},Z-Z_{K,k})\\ 
    +\frac{1}{2} \Big(
    J'(Y_{K,k})(Y-i_{\X}Y)-A'_k(Y_{K,k})(Z_{K,k},Y-i_{\X} Y)
    - A_{k}(Y_{K,k},Z-i_{\Y}Z)
    \Big)\\
    \qquad 
    +\frac{1}{2}\Big(
    [A'_k-A'](Y_{K,k})(Z_{K,k},Y-Y_{K,k})+
    [A_k-A](Y_{K,k},Z+Z_{K,k})    \Big),
  \end{multline}
  where ${\cal R}_K^{(3)}$ is of third order in the primal and adjoint
  discretization error.
  %
\end{lemma}
\begin{proof}
  We introduce one Lagrangian for the continuous and the semidiscrete
  model and one for the fully discrete model
  \[
  L(Y,Z)\coloneqq J(Y)-A(Y,Z),\quad L_k(Y,Z)\coloneqq J(Y)-A_k(Y,Z). 
  \]
  For the solutions $Y\in \X$ and $Y_{K,k}\in \X_K$ it holds for all
  $Z\in \Y$ and $Z_K\in \Y_K$ that
  \begin{equation}\label{DWR:2}
    \begin{aligned}
      J(Y)-J(Y_{K,k}) &= L(Y,Z)-L_k(Y_{K,k},Z_K)\\
      &= \underbrace{L(Y,Z)-L(Y_{K,k},Z_K)}_{=\Ea}+
      \underbrace{L(Y_{K,k},Z_K)-L_k(Y_{K,k},Z_K)}_{=\Eb}. 
    \end{aligned}
  \end{equation}
  The first part \Ea\ is a standard DWR term, which by 
  defining $\xt\coloneqq(Y,Z)$ and $\xt_{K,k}\coloneqq(Y_{K,k},Z_K)$, 
  is approximated by writing the difference as integral over its 
  derivative and by approximation with the trapezoidal rule, compare \cite[Proposition 2.1]{BeckerRannacher2001}
  \[
  \Ea=L(\xt)-L(\xt_{K,k}) = \frac{1}{2}\Big(L'(\xt)(\xt-\xt_{K,k})
  +L'(\xt_{K,k})(\xt-\xt_{K,k})\Big) + {\cal R}^{(3)}_K(\xt-\xt_{K,k}),
  \]
  where the remainder ${\cal R}_K^{(3)}(\xt-\xt_{K,k})$ is of third order in the error. The relation
  \[
  L'(\xt)(\delta \xt) = J'(Y)(\delta Y)-A'(Y,Z)(\delta Y) - A(Y,\delta
  Z) 
  \]
  shows that it holds $L'(\xt)(\delta \xt)=L'(Y,Z)(\delta Y,\delta Z)=0$
  for the analytical solutions $Y,Z\in \X\times \Y$ and for all $\delta\xt=(\delta Y,\delta Z)\in \X\times \Y$. We neglect the remainder ${\cal R}^{(3)}$ and approximate
  \begin{equation}\label{DWR:3}
    \Ea \approx \frac{1}{2} \Big(
    J'(Y_{K,k})(Y-Y_{K,k})-A'(Y_{K,k})(Z_{K,k},Y-Y_{K,k})
    - A(Y_{K,k},Z-Z_K)\Big).
  \end{equation}
  The forms $A(\cdot,\cdot)$ and $A'(\cdot)(\cdot,\cdot)$ are based on
  the non-discrete transfer operator ${\cal F}$. The discrete primal
  and adjoint solutions $Y_{K,k}\in 
  \X_K$ and $Z_K\in \Y_K$ are however defined by
  using the discrete form $A_k(\cdot,\cdot)$. We insert
  $\pm A_k$ and $\pm A'_k$ in~(\ref{DWR:3}) such  that we can apply
  Galerkin orthogonality to introduce interpolations $i_\X:\X\to \X_K$
  and $i_\Y:\Y\to \Y_K$
  \begin{equation}\label{DWR:4}
    \begin{aligned}
      \Ea &\approx \frac{1}{2} \Big(
      J'(Y_{K,k})(Y-i_{\X}Y)-A'_k(Y_{K,k})(Z_{K,k},Y-i_{\X} Y)
      - A_k(Y_{K,k},Z-i_{\Y}Z)
      \Big)\\
      &\qquad 
      +\frac{1}{2}\Big(
      [A'_k-A'](Y_{K,k})(Z_{K,k},Y-Y_{K,k})+
      [A_k-A](Y_{K,k},Z-Z_K)    \Big). 
    \end{aligned}
  \end{equation}
  The notation $[A_k-A](Y,Z)\coloneqq A_k(Y,Z)-A(Y,Z)$ is introduced
  for brevity.
  In~(\ref{DWR:2}), the second error component \Eb\ is a conformity
  error and given as
  \[
  \Eb = [A_k-A](Y_{K,k},Z_K). 
  \]
  Together with~(\ref{DWR:4}) we obtain the postulated result.  
\end{proof}
The second line of~(\ref{error1}) is the standard residual representation of the DWR error estimator. Given a reconstruction of the \emph{weights} $Y-i_\X Y$ and $Z-i_\Y Z$ it can be evaluated numerically, we refer to Section~\ref{sec:evaluation} for details. The third line in~(\ref{error1}) combines two conformity errors coming from the replacement of the transfer operator ${\cal F}$ by its discrete counterpart ${\cal F}_k$. These terms will be discussed in the following paragraphs. 
\begin{lemma}[Primal conformity error]
  \label{lemma:primal}
  Let the assumptions of Theorem~\ref{theorem:dwr} hold. For the
  primal conformity error it holds
  \[
    [A_k-A](Y_{K,k},Z+Z_{K,k})=\epsilon
    \int_0^T \eta^\pi\big(Y_{K,k}(t)\big)\cdot
    \big(Z(t)+Z_{K,k}(t)\big)\,\text{d}t
    +\epsilon \int_0^T {\cal R}_k^{(3)}\big(Y_{K,j}(t)\big)\,\text{d}t,
    \]
    where $\eta^\pi(Y)$ is defined in~(\ref{micro:eta}) and where ${\cal R}_k^{(3)}$ is a remainder of third order in the error. 
\end{lemma}
\begin{proof}
  For ease of notation we introduce $\tilde Z\coloneqq Z+Z_{K.k}$. This term does not carry any convergence properties and the full order of convergence must be reconstructed from the difference of the two forms $[A_k-A]$. 
  Subtracting~(\ref{var:long}) from~(\ref{vardisc:long}) gives
  \begin{equation}\label{DWR:2a}
    [A_k-A](Y_{K,k},\tilde Z)
    =\epsilon \sum_{n=1}^N\int_{T_{n-1}}^{T_n}
    \Big({\cal F}\big(Y_{K,k}(t)\big)
    -{\cal F}_k\big(Y_{K,k}(t)\big) \Big)\tilde Z(t).
  \end{equation}
  For the evaluation of this term we use a nested application of the DWR estimator, since the difference between the fast scale influences $u_Y$ and $u_{k;Y}$ enters implicitly.\footnote{A practical evaluation of this error term will require numerical quadrature of the integrals on the right hand side, e.g. by the 2-point Gauss rule. The error term ${\cal F}\big(Y_{K,k}(t)\big)-{\cal F}_k\big(Y_{K,k}(t)\big)$ must hence be approximated in two points in each time step $[T_{n-1},T_n]$.}
  
  Now, let $t\in [0,T]$ be fixed and $Y\coloneqq Y_{K,k}(t)$. We introduce
  \begin{equation}\label{DWR:2d}
    J^\pi(u_Y)\coloneqq \int_0^1 f\big(Y,u_Y(s)\big)\,\text{d}s
  \end{equation}
  such that ${\cal F}\big(Y_{K,k}(t)\big)-{\cal F}_k\big(Y_{K,k}(t)\big)=J^\pi(u_Y)-J^\pi(u_{k;Y})$,  where by $u_Y$ we denote the continuous periodic micro solution to~(\ref{var:fast}) and by $u_{k;Y}$ the discrete solution to~(\ref{vardisc:fast}), which satisfies the periodicity approximately, i.e. $\|u_{k;Y}(1)-u_{k;Y}(0)\|<tol_P$.  With the solution $z_{k;Y}\in \W_k$ to the adjoint micro problem
  \[
  B'(u_{k;Y})(\psi_k,z_{k;Y}) = J^{\pi'}(u_Y)(\psi_k)\quad\forall \psi_k\in \V_k^\pi 
  \]
  we estimate in the usual DWR way, see Lemma~\cite[Section 2.3]{BeckerRannacher2001},
  \begin{multline}\label{DWR:2e}
    J^\pi(u_Y)-J^\pi(u_{k;Y})= 
    {\cal R}^{(3)}_k(u_{k;Y})(u_Y-u_{k;Y},    z_Y-z_{k;Y}) 
    -\frac{1}{2}B(u_{k;Y}, z_Y-i_\W z_Y)\\
    + \frac{1}{2}\Big(J^{\pi'}(u_{k;Y})(u_Y-i_\V u_Y)-
    B'(u_{k;Y})(u_Y-i_\V u_Y, z_{k;Y})
    \Big)
    + {\cal O}(tol_P),
  \end{multline}
  where the $tol_P$-term arises from the disturbed Galerkin orthogonality. To clarify the impact of the approximated tolerance we give a sketch: assume that $u_{k;Y}^\pi(t)$ is the fully periodic solution, strictly satisfying $u_{k;Y}^\pi(1)=u_{k;Y}^\pi(0)$. Then, 
  \[
  \begin{aligned}
    B(u_{k;Y};\phi_{K,k}) &=
    \underbrace{B(u_{k;Y}^\pi,\phi_{K,k})}_{=0}
    +B(u_{k;Y}-u_{k;Y}^\pi,\phi_{K,k})\\
    \Rightarrow\quad
    \big|B(u_{k;Y},\phi_{K,k})\big| &\le 
    c\, \|u_{k;Y}-u^\pi_{k;Y}\|_{L^\infty(I^P)}\,
    \|\phi_{K,k}\|_{L^\infty(I^P)}\le 
    c\,tol_P \|\phi_{K,k}\|_{L^\infty(I^P)},
  \end{aligned}
  \]
  where we estimated the periodicity error $\|u_{k;Y}^\pi-u_{k;Y}\|\le tol_P$ by the imposed threshold. 
  
  Details on the adjoint solution  $z_Y=z_Y\in \W$ and its discretization $z_{k;Y}\in \W_k$  entering~(\ref{DWR:2e}) are discussed in Section~\ref{sec:evaluation}. 
\end{proof}

\begin{remark}
  \label{remark:adjointconsistency}
  The adjoint consistency error arising in  Lemma~\ref{dwr:long} can be estimated as 
  \begin{multline}\label{dc:1}
    \Big|[A'_k-A'](Y_{K,k})(Y-Y_{K,k},Z_{K,k})\Big|\\
    \le \epsilon\int_0^T
    \Big|\Big( {\cal F'}_k\big(Y_{K,k}(t)\big)
    - {\cal F'}\big(Y_{K,k}(t)\big)\Big)\cdot  
    \big(Y(t)-Y_{K,k}(t)\big)\cdot Z_{K,k}(t)\Big|\,\text{d}t\\
    \le \epsilon T \|Z_{K,k}\|_{L^\infty(I)}
    \| {\cal F}'_k\big(Y_{K,k}\big)-{\cal F}'\big(Y_{K,k}\big)\|_{L^\infty(I)}
    \| Y-Y_{K,k} \|_{L^\infty(I)}
  \end{multline}
  In~\cite{FreiRichter2020} we have shown a second order error estimator for the primal error in a comparable multiscale setting
  \[
  \| Y-Y_{K,k} \|_{L^\infty(I)} = {\cal O}\Big(  \epsilon + \epsilon^2K^2 + k^2\Big). 
  \]
  The first term in~(\ref{dc:1}) is bounded, since the adjoint
  problem $A'(Y)(\Phi,Z)=J'(Y)(\Phi)$, going backward in time, is equivalent to
  \[
  -Z'(t) - \epsilon {\cal F}'\big(Y(t)\big)Z(t) = 0,\quad Z(T)=1
  \quad\Rightarrow\quad
  Z(t) =\exp\Big(\epsilon\int_T^t{\cal
    F}'\big(Y(s)\big)\,\text{d}s\Big),
  \]
  which is bounded for $0\le \epsilon t\le\epsilon T={\cal O}(1)$, since the adjoint transfer operator ${\cal F'}(Y)$ is bounded. The remaining term in~(\ref{dc:1}) measures the discretization error in the adjoint fast scale problem. With arguments similar to those used in the proof to Lemma~\ref{lemma:primal}, second order convergence in $k$ can be shown. Overall, the adjoint consistency error is of higher order
  $
  \Big|[A'_k-A'](Y_{K,k})(Y-Y_{K,k},Z_{K,k})\Big|
  ={\cal O}\Big(\epsilon k^2+\epsilon^2 k^2K^2+k^4\Big).
  $
\end{remark}
%


\subsection{Adjoint problems and evaluation of the error estimator} 
\label{sec:evaluation}

The error estimator~(\ref{errorestimate}) depends on the adjoint solution $Z\in \Y$ and also on the adjoint micro scale solutions $z_Y\in \W$. We shortly sketch the steps required to approximate these adjoint solutions, as the multiscale framework will require a nested approach. For $Y\in \X$ given, $Z\in \Y$ is defined as solution to 
\begin{equation}\label{de:1}
  \int_0^T \big(\Psi'(t)-\epsilon \nabla_Y{\cal F}\big(Y(t)\big)(\Psi(t))\big)\cdot Z(t)\,\text{d}t = J'(Y)(\Psi)\quad\forall \Psi\in \X. 
\end{equation}
The derivative of the transfer operator is given by
\begin{equation}\label{de:2}
  \nabla_Y {\cal F}(Y)(\Psi) = \int_0^1
  \Big(\nabla_Y f\big(Y(t),u_Y(s)\big)+
  \nabla_u f\big(Y(t),u_Y(s)\big)\big(D_Yu_Y(s)\big)\Big)\Psi(t)\,\text{d}s. 
\end{equation}
While the first part involving the derivative in direction of $Y$ is directly accessible, evaluation of the second term requires a further tangent solution $D_y u_Y\in \V^\pi$, the derivative of the time periodic solution $u_Y(s)$ with respect to $Y$. It is given as the solution to
\begin{equation}\label{de:3}
  B'(u_Y)(D_Yu_Y,\phi) = 0\quad\forall\phi\in \W. 
\end{equation}
Finally, to estimate the conformity error introduced by replacing the transfer operator $\cal F$ by its discrete counterpart a further adjoint micro scale solution $z_Y\in \W$, given by the following periodic in time problem must be solved
\begin{equation}\label{de:3.1}
  B'(u_Y)(\psi,z_Y) = J^{\pi'}(u_Y)(\psi)\quad \forall \psi\in \V^\pi,
\end{equation}
where $J^\pi(u_Y) =\int_0^1 f(Y,u_Y(s))\,\text{d}s$.

\begin{figure}[t]
  \begin{center}
    \includegraphics[width=0.49\textwidth]{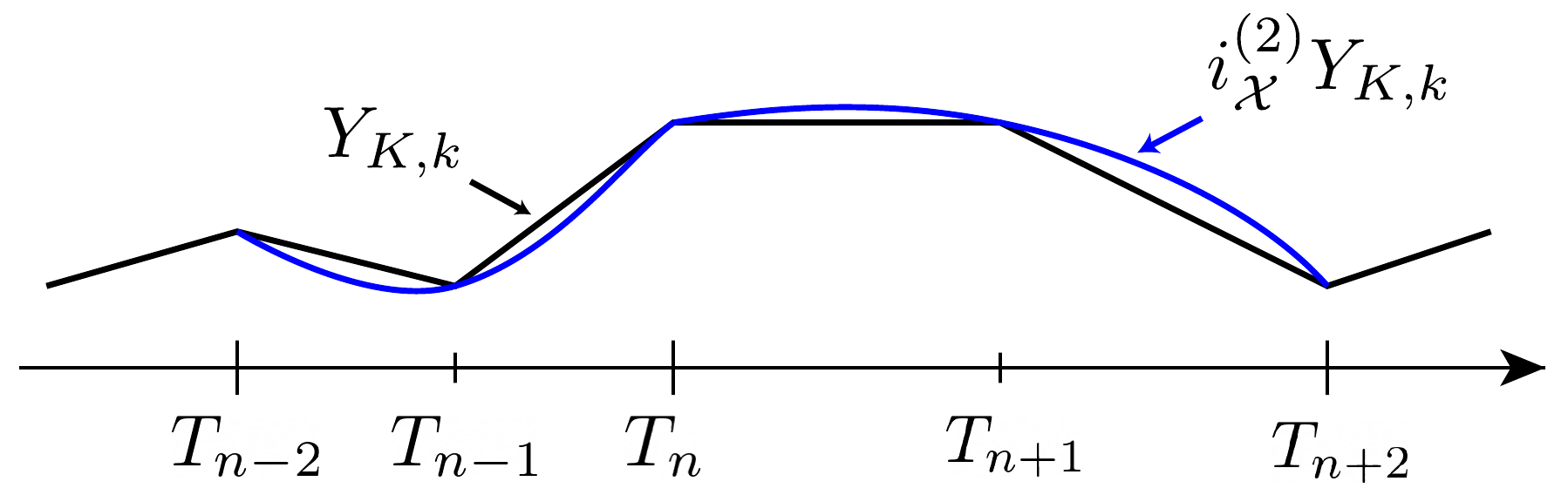}
    \includegraphics[width=0.49\textwidth]{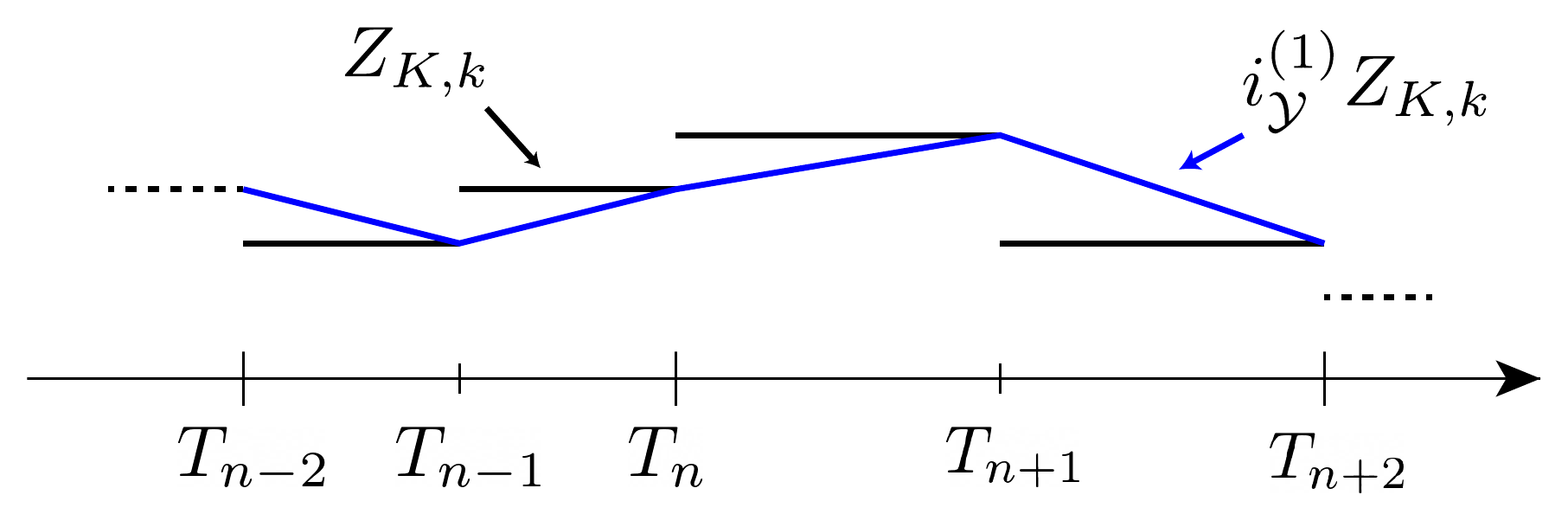}
  \end{center}
  \caption{Reconstructing a higher order approximation from the discrete
    solutions. By $i_\X^{(2)}$ (left) we piecewise quadratic reconstruction on the mesh with twice the mesh spacing and by $i_\Y^{(1)}$ (right) the linear reconstruction on the same mesh.}
  \label{fig:weights}
\end{figure}

The a posteriori error estimator presented in Theorem~\ref{theorem:dwr} cannot be evaluated exactly since it depends on the unknown exact solutions $Y\in \X$ and $Z\in \Y$. Further, several higher order remainders appear, which are simply omitted. To approximate primal and dual residuals weights $Y-i_\X Y$ and $Z-i_\Y Z$ and also to approximate the sum of continuous and discrete adjoint solution $Z+Z_{K,k}$ we use the usual reconstruction mechanism that is based on computing $Y_{K,k}\in \X_K$ and $Z_{K,k}\in \Y_K$ and applying a higher order interpolation by reinterpreting the piecewise linear function $Y_{K,k}$ as piecewise quadratic and the piecewise constant function $Z_{K,k}$ as piecewise linear. Fig.~\ref{fig:weights} illustrates this procedure. We ensure that all macro meshes have a patch structure: two adjacent intervals $I_{2n-1}$ and $I_{2n}$ each have the size $K_{2n-1}=K_{2n}$. The micro meshes are uniform.

This reconstruction of the weights must be considered a computational tool for approximating the functional error. It cannot however give rigorous upper and lower bounds. For general details on this reconstruction we refer to~\cite{BeckerRannacher2001,RichterWick2015_dwr} and in particular to~\cite{MeidnerRichter2014} in the context of temporal Galerkin schemes.

\section{Numerical examples}
\label{sec:num}

\begin{problem}\label{problem:numex1}
  On $I=[0,T]$ with $T=6\cdot 10^5$ find  $y\in C^1(I)$ and $u\in C^2(I)$ such that 
  \begin{equation}\label{testeq:1}
    \begin{aligned}
      &y'(t) = \epsilon f\big(y(t),u(t)\big),& y(0)=0,\\
      &u''(t) + \frac{1}{2} u'(t) + \gamma(y(t))u(t) = \sin(2\pi t), &u(0)=u_0,\, u'(0)=u_0',
    \end{aligned}
  \end{equation}
  with the scale separation parameter $\epsilon=10^{-6}$ and
\begin{equation}\label{reaction}
  f(y,u)\coloneqq \frac{1}{(1+y)(1+64 u^2)},\quad 
  \gamma(y)\coloneqq \big(4\pi^2+32(y-1)\big).
\end{equation}
As functional of interest we consider the slow scale component at final time $T$
\[
J(y) = y(T).
\]
\end{problem}

We produce reference values for the functional output $J(y)=y(T)$ by resolved simulations based on a direct discretization of Problem~\ref{problem:numex1} with the trapezoidal rule using  a small time step size $k$ over the full period of time $I=[0,T]$. Extrapolating $k\to 0$ shows the experimental order of convergence ${\cal O}(k^{2.0015})$ and for all further comparisons we set the reference value to 
\begin{equation}\label{referencevalue}
  J(y_{ref})\coloneqq  1.08704164. 
\end{equation}


We first show that this problem fits into the framework introduced in
Section~\ref{sec:modelproblem}.
\begin{lemma}\label{lemma:problemass}
  Problem~\ref{problem:numex1} satisfies
  Assumptions~\ref{assumption:f} and~\ref{assumption:g}. 
\end{lemma}
\begin{proof}
  For $f(t,y)$ it holds
  $|f(t,y)|\le 1$ and, for $y,Y\in\mathds{R}^c$
  \[
  \big|f(y,u)-f(Y,u)\big| = \frac{|Y-y|}{|1+y|\cdot|1+Y|\cdot |1+64
    u^2|}\le |y-Y|, 
  \]
  as well as
  \[
  \big|f(y,u)-f(y,U)\big| = \frac{64|u+U|\cdot |u-U|}{|1+y|\cdot|1+64
    u^2|\cdot |1+64U^2|}
  \le 8 |u-U|,
  \]
  which shows boundedness and Lipschitz continuity with $\Cas{1Bf}=1$ and $\Cas{1Lf}=8$. 
  
  Next we reformulate the microscale problem as a first order system in
  $v(t):=\big(u_1(t),u_2(t)\big)$, with $u_1(t) = u(t)$ and $u_2(t)=u'(t)$
  \begin{equation}\label{l19:1}
    \begin{pmatrix}u_1\\u_2\end{pmatrix}'(t)=
      \begin{pmatrix}
        0 & 1 \\ - \gamma(y) & -\frac{3}{5}
      \end{pmatrix}\begin{pmatrix}u_1(t)\\u_2(t)
      \end{pmatrix}
      +
      \begin{pmatrix} 0\\ \sin(2\pi t) \end{pmatrix} =: G\big(t,y,v(t)\big)
  \end{equation}
  The Jacobian $\nabla_u G$ has the eigenvalues
  $\lambda_{1/2} = -\frac{3}{10} \pm
  \frac{\sqrt{9-100\gamma(y)}}{10}$, 
  which for the specific choice of $\gamma(y)$ have strictly negative
  real part with $\Las{negdef}=\frac{3}{10}$. This shows that there
  exists a unique solution which is bound by the initial values and
  the right hand side $g(t)=\sin(2\pi t)$
  \begin{equation}\label{l19:1.5}
    \|v(t)\| \le C \left(\|v(0)\| + \frac{1}{\Las{negdef}}\right),
  \end{equation}
  with a constant $C=C(y)$ which depends on the eigenvectors of
  $\nabla_u G$, hence on  $y\in X^c$. 
  Furthermore, for $y\in X^c$
  fixed there exists a unique periodic solution, since, for each
  initial value the difference $v(t+1)-v(t)$ will decay to zero. We
  will denote such a periodic solution by $v_y(t)$. And since this
  solution is also reached for the initial $v(0)=0$,
  estimate~(\ref{l19:1.5}) gives
  \begin{equation}\label{l19:1.6}
    \|v_y(t)\| \le C \frac{1}{\Las{negdef}}. 
  \end{equation}
  Having two such periodic solution $v_1(t)$ and $v_2(t)$ belonging to
  the parameters $y_1,y_2\in X^c$ satisfy
  \[
  v_1'(t)-v_2'(t) =
  \begin{pmatrix}
    0&1\\-\gamma(y_1)&-\frac{3}{5}
  \end{pmatrix}\big(v_1(t)-v_2(t)\big) +
  \begin{pmatrix}0&0\\\gamma(y_1)-\gamma(y_2)&0
  \end{pmatrix}v_2(t).
  \]
  Considered as equations for the difference $v_1(t)-v_2(t)$ this
  corresponds to problem~(\ref{l19:1}) with $g(t)$ replaced  by the
  periodic function $\big(\gamma(y_1)-\gamma(y_2)\big)v_2(t)$ such
  that~(\ref{l19:1.6}) yields
  \[
  \|v_1(t)-v_2(t)\| \le \frac{C}{\Las{negdef}} \sup_{t\in I^P}
  \|v_2(t)\| \cdot \|\gamma(y_1)-\gamma(y_2)\|
  \le \frac{32C^2}{\Las{negdef}^2}  \|y_1-y_2\|. 
  \]
\end{proof}

\subsection{Convergence of the multiscale algorithm}

We start by analyzing the convergence of the multiscale scheme by running simulations with different but uniform time step sizes for $K$ and $k$ specified by
\begin{equation}\label{stepsizes}
  K_i = 100\,000\cdot 2^{-i},\quad
  k_j = 0.1\cdot 2^{-j},\quad i,j\in \{0,1,2,3,4,5\}. 
\end{equation}
In Fig.~\ref{fig:error_convergence} we show convergence with respect to the small time step size $k$ (left) and with respect to the large time step size $K$ (right). In both cases second order convergence is obtained as long as the step size under investigation is dominant. Furthermore, the results show that the range of chosen step sizes~(\ref{stepsizes}) is balanced with a slight dominance of the small scale error depending on $k$. The raw data is also given in Table~\ref{table:DWR1}. 

\begin{figure}[t]
  \centering
  \begin{subfigure}[b]{0.45\textwidth}
    \includegraphics[width=\textwidth]{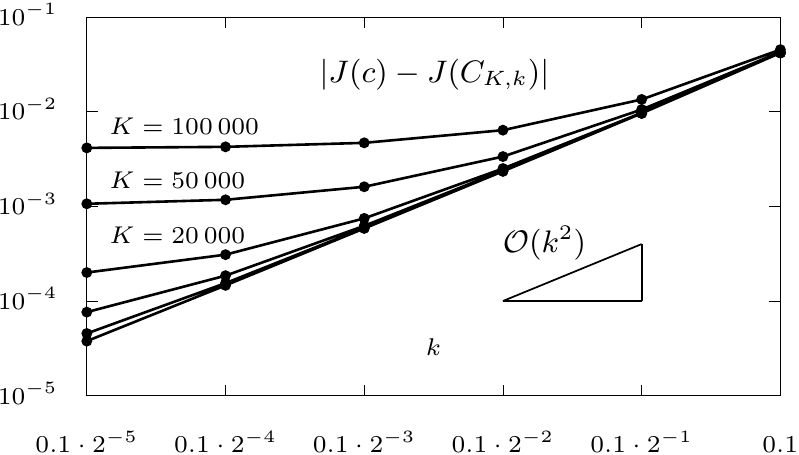}
    \caption{Refinement of the micro scale step $k$.}
    \label{fig:error_k}
  \end{subfigure}
  ~ 
  \begin{subfigure}[b]{0.45\textwidth}
    \includegraphics[width=\textwidth]{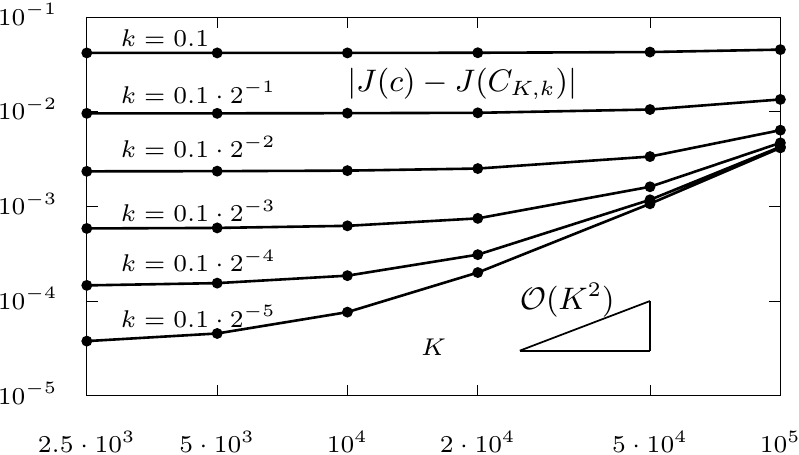}
    \caption{Refinement of the macro scale step $K$.}
    \label{fig:error_K}
  \end{subfigure}
  \caption{Error with respect to the small step size $k$ and the large step size $K$. Left: each line represents a fixed value of $K$. Right: each line takes $k$ fixed.}
  \label{fig:error_convergence}
\end{figure}

\subsection{Evaluation of the error estimator}

\begin{figure}[t]
  \centering
  \includegraphics[width=.45\textwidth]{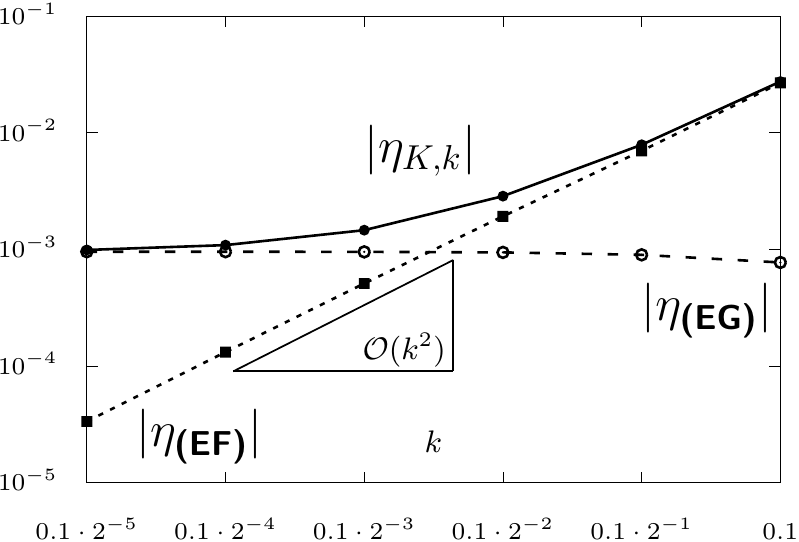}
  \includegraphics[width=.45\textwidth]{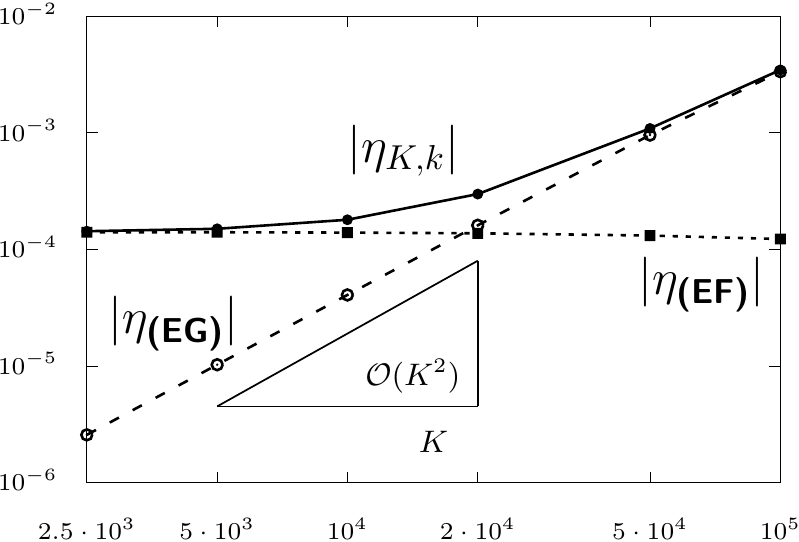}
  \caption{Error estimator $\eta_{K,k}$ and the parts $\eta_\EG$ and
    $\eta_\EF$ that make it up. Left: fixed macro step size $K=50\,000$,
    and right: fixed micro step size $k=0.00625$.}
  \label{fig:dwr_parts}
\end{figure}

Next, we analyze the quality of the a posteriori error estimator derived in the previous section. We will show that this error estimator is accurate in predicting $J(y)-J(Y_{K,k})$ for the complete range of step sizes shown in~(\ref{stepsizes}). The smallest step sizes reach $K_{min}=2\,500$ and $k_{min}=0.003125$ such that the discretization error $O(\epsilon^2K^2+k^2)$ will still dominate the averaging error $O(\epsilon)$. The tolerance for reaching periodicity is set to $tol_P=10^{-9}$. The raw values indicating $J(y_{ref})-J(Y_{K,k})$, overall error estimator $\eta_{K,k}$ and its splitting into discretization error $\eta_\EG$, primal conformity error (fine scale error) $\eta_\EF$ adjoint conformity error (which is of higher order) $\eta_{\EF'}$ are shown in Table~\ref{table:DWR1}. Finally, we also present the efficiency of the error estimator by indicating the effectivity index 
\begin{equation}\label{effectivity}
  \text{eff}_{K,k} = 100\%\cdot\frac{\eta_{K,k}}{J(y_{ref})-J(Y_{K,k})}. 
\end{equation}
Values above $100\%$ show an overestimation of the error, values below $100\%$ an underestimation. The results collected in Table~\ref{table:DWR1} however show  a highly robust estimation for all combinations of small and large step sizes.

\begin{table}[t]
  \resizebox{\textwidth}{!}{\begin{tabular}{c | c | c  c | c  c  c | c }
      \toprule
      k & K &\small $J(y_{ref})-J(Y_{K,k})$ & $\eta_{K,k}$& $\eta_{\EG}$ & $\eta_{\EF}$ & $\eta_{\EF'}$ & $\text{eff}_{K,k}$\\
      \midrule
      \multirow{6}{*}{0.1}
      &100000&$4.52\cdot 10^{-2}$&$2.81\cdot 10^{-2}$&$2.74\cdot 10^{-3}$&$2.55\cdot 10^{-2}$&$-1.77\cdot 10^{-4\phantom{0}}$&62.1\%\\
      &50000&$4.26\cdot 10^{-2}$&$2.75\cdot 10^{-2}$&$7.71\cdot 10^{-4}$&$2.68\cdot 10^{-2}$&$-3.36\cdot 10^{-5\phantom{0}}$&64.6\%\\
      &20000&$4.19\cdot 10^{-2}$&$2.77\cdot 10^{-2}$&$1.30\cdot 10^{-4}$&$2.76\cdot 10^{-2}$&$-5.30\cdot 10^{-6\phantom{0}}$&66.3\%\\
      &10000&$4.18\cdot 10^{-2}$&$2.79\cdot 10^{-2}$&$3.29\cdot 10^{-5}$&$2.79\cdot 10^{-2}$&$-1.23\cdot 10^{-6\phantom{0}}$&66.9\%\\
      &5000&$4.17\cdot 10^{-2}$&$2.81\cdot 10^{-2}$&$8.28\cdot 10^{-6}$&$2.81\cdot 10^{-2}$&$-3.07\cdot 10^{-7\phantom{0}}$&67.2\%\\
      &2500&$4.17\cdot 10^{-2}$&$2.81\cdot 10^{-2}$&$2.08\cdot 10^{-6}$&$2.81\cdot 10^{-2}$&$-7.67\cdot 10^{-8\phantom{0}}$&67.4\%\\
      \midrule
      \multirow{6}{*}{0.05}
      &100000&$1.34\cdot 10^{-2}$&$9.64\cdot 10^{-3}$&$3.12\cdot 10^{-3}$&$6.52\cdot 10^{-3}$&$-9.50\cdot 10^{-6\phantom{0}}$&71.7\%\\
      &50000&$1.06\cdot 10^{-2}$&$7.90\cdot 10^{-3}$&$8.99\cdot 10^{-4}$&$7.00\cdot 10^{-3}$&$-2.56\cdot 10^{-6\phantom{0}}$&74.8\%\\
      &20000&$9.74\cdot 10^{-3}$&$7.45\cdot 10^{-3}$&$1.51\cdot 10^{-4}$&$7.30\cdot 10^{-3}$&$-2.70\cdot 10^{-7\phantom{0}}$&76.5\%\\
      &10000&$9.63\cdot 10^{-3}$&$7.44\cdot 10^{-3}$&$3.83\cdot 10^{-5}$&$7.41\cdot 10^{-3}$&$-6.62\cdot 10^{-8\phantom{0}}$&77.3\%\\
      &5000&$9.60\cdot 10^{-3}$&$7.47\cdot 10^{-3}$&$9.64\cdot 10^{-6}$&$7.46\cdot 10^{-3}$&$-1.63\cdot 10^{-8\phantom{0}}$&77.8\%\\
      &2500&$9.59\cdot 10^{-3}$&$7.48\cdot 10^{-3}$&$2.42\cdot 10^{-6}$&$7.48\cdot 10^{-4}$&$-4.05\cdot 10^{-9\phantom{0}}$&78.0\%\\
      \midrule
      \multirow{6}{*}{0.025}
      &100000&$6.39\cdot 10^{-3}$&$5.07\cdot 10^{-3}$&$3.28\cdot 10^{-3}$&$1.78\cdot 10^{-3}$&$2.80\cdot 10^{-6\phantom{0}}$&79.3\%\\
      &50000&$3.36\cdot 10^{-3}$&$2.86\cdot 10^{-3}$&$9.41\cdot 10^{-4}$&$1.92\cdot 10^{-3}$&$7.57\cdot 10^{-8\phantom{0}}$&85.2\%\\
      &20000&$2.51\cdot 10^{-3}$&$2.16\cdot 10^{-3}$&$1.58\cdot 10^{-4}$&$2.00\cdot 10^{-3}$&$3.70\cdot 10^{-8\phantom{0}}$&86.3\%\\
      &10000&$2.38\cdot 10^{-3}$&$2.07\cdot 10^{-3}$&$4.01\cdot 10^{-5}$&$2.03\cdot 10^{-3}$&$1.04\cdot 10^{-8\phantom{0}}$&87.0\%\\
      &5000&$2.35\cdot 10^{-3}$&$2.06\cdot 10^{-3}$&$1.01\cdot 10^{-5}$&$2.05\cdot 10^{-3}$&$2.67\cdot 10^{-9\phantom{0}}$&87.5\%\\
      &2500&$2.35\cdot 10^{-3}$&$2.06\cdot 10^{-3}$&$2.53\cdot 10^{-6}$&$2.05\cdot 10^{-3}$&$6.71\cdot 10^{-10}$&87.7\%\\
      \midrule
      \multirow{6}{*}{0.0125}
      &100000&$4.68\cdot 10^{-3}$&$3.80\cdot 10^{-3}$&$3.32\cdot 10^{-3}$&$4.74\cdot 10^{-4}$&$1.12\cdot 10^{-6\phantom{0}}$&81.1\%\\
      &50000&$1.61\cdot 10^{-3}$&$1.46\cdot 10^{-3}$&$9.53\cdot 10^{-4}$&$5.09\cdot 10^{-4}$&$9.70\cdot 10^{-8\phantom{0}}$&90.9\%\\
      &20000&$7.46\cdot 10^{-4}$&$6.92\cdot 10^{-4}$&$1.60\cdot 10^{-4}$&$5.32\cdot 10^{-4}$&$2.22\cdot 10^{-8\phantom{0}}$&92.7\%\\
      &10000&$6.23\cdot 10^{-4}$&$5.79\cdot 10^{-4}$&$4.06\cdot 10^{-5}$&$5.39\cdot 10^{-4}$&$5.86\cdot 10^{-9\phantom{0}}$&93.1\%\\
      &5000&$5.92\cdot 10^{-4}$&$5.53\cdot 10^{-4}$&$1.02\cdot 10^{-5}$&$5.43\cdot 10^{-4}$&$1.48\cdot 10^{-9\phantom{0}}$&93.4\%\\
      &2500&$5.84\cdot 10^{-4}$&$5.47\cdot 10^{-4}$&$2.56\cdot 10^{-6}$&$5.44\cdot 10^{-4}$&$3.72\cdot 10^{-10}$&93.7\%\\
      \midrule
      \multirow{6}{*}{0.00625}
      &100000&$4.26\cdot 10^{-3}$&$3.46\cdot 10^{-3}$&$3.33\cdot 10^{-3}$&$1.22\cdot 10^{-4}$&$3.16\cdot 10^{-7\phantom{0}}$&81.2\%\\
      &50000&$1.17\cdot 10^{-3}$&$1.09\cdot 10^{-3}$&$9.55\cdot 10^{-4}$&$1.31\cdot 10^{-4}$&$3.37\cdot 10^{-8\phantom{0}}$&92.5\%\\
      &20000&$3.09\cdot 10^{-4}$&$2.98\cdot 10^{-4}$&$1.60\cdot 10^{-4}$&$1.37\cdot 10^{-4}$&$7.13\cdot 10^{-9\phantom{0}}$&96.2\%\\
      &10000&$1.86\cdot 10^{-4}$&$1.80\cdot 10^{-4}$&$4.07\cdot 10^{-5}$&$1.39\cdot 10^{-4}$&$1.86\cdot 10^{-9\phantom{0}}$&96.9\%\\
      &5000&$1.55\cdot 10^{-4}$&$1.50\cdot 10^{-4}$&$1.02\cdot 10^{-5}$&$1.40\cdot 10^{-4}$&$4.70\cdot 10^{-10}$&97.2\%\\
      &2500&$1.47\cdot 10^{-4}$&$1.43\cdot 10^{-4}$&$2.57\cdot 10^{-6}$&$1.40\cdot 10^{-4}$&$1.18\cdot 10^{-10}$&97.4\%\\
      \midrule
      \multirow{6}{*}{0.003125}
      &100000&$4.15\cdot 10^{-3}$&$3.37\cdot 10^{-3}$&$3.34\cdot 10^{-3}$&$3.11\cdot 10^{-5}$&$8.27\cdot 10^{-8\phantom{0}}$&81.2\%\\
      &50000&$1.07\cdot 10^{-3}$&$9.90\cdot 10^{-4}$&$9.56\cdot 10^{-4}$&$3.34\cdot 10^{-5}$&$9.58\cdot 10^{-9\phantom{0}}$&92.8\%\\
      &20000&$2.00\cdot 10^{-4}$&$1.95\cdot 10^{-4}$&$1.61\cdot 10^{-4}$&$3.48\cdot 10^{-5}$&$1.98\cdot 10^{-9\phantom{1}}$&97.6\%\\
      &10000&$7.64\cdot 10^{-5}$&$7.60\cdot 10^{-5}$&$4.07\cdot 10^{-5}$&$3.53\cdot 10^{-5}$&$5.14\cdot 10^{-10}$&99.5\%\\
      &5000&$4.55\cdot 10^{-5}$&$4.58\cdot 10^{-5}$&$1.03\cdot 10^{-5}$&$3.55\cdot 10^{-5}$&$1.30\cdot 10^{-10}$&100.7\%\\
      &2500&$3.77\cdot 10^{-5}$&$3.82\cdot 10^{-5}$&$2.57\cdot 10^{-6}$&$3.57\cdot 10^{-5}$&$3.25\cdot 10^{-11}$&101.3\%\\
      \bottomrule
  \end{tabular}}
  \caption{Functional error $J(Y_{K,k})-J(y_{ref})$ and error estimator $\eta_{K,k}$ for different step sizes for the large scale and small scale problem. $\eta_\EG$, $\eta_\EF$ and $\eta_{\EF'}$ are the contributions of the estimator $\text{eff}_{K,k}$ the effectivity, compare~(\ref{effectivity}). All values are rounded to the first three relevant digits.}\label{table:DWR1} 
\end{table}

\FloatBarrier

The analysis of the different contribution shows that $\eta_\EG$ indicates the long term error depending mostly on $K$ and $\eta_\EF$ indicates the short term error depending on $k$, both converging with order two. The adjoint consistency error $\eta_{\EF'}$ shows higher order convergence $O(\epsilon k^2 + \epsilon^2k^2K^2+k^4)$ as stated in Remark~\ref{remark:adjointconsistency} and hence, it can be neglected. 

For $K=50\,000$ fixed and varying $k$ and for $k=0.00625$ fixed and varying $K$ respectively, Figure~\ref{fig:dwr_parts} shows the separation of the error estimator into long term and short term influences each. These results motivate to use $\eta_\EF$ and $\eta_\EG$ for controlling an adaptive procedure to find an optimally balanced discretization $K,k$.

\subsection{Adaptive control}

We write the error estimator as a sum over the subdivisions of the long time horizon $I$. By doing this we can quantify the error contribution of each subdivision. For the error contribution of element $I_n=(T_{n-1},T_n]$, the error contribution is comprised of two parts. The error discretization of the averaged long term problem

\begin{multline}\label{ls_subdivision}
  \eta_{\EG}^n:= -\frac{1}{2} A_{k|I_n}(Y_{K,k},Z-i_Y Z) \\
  +\frac{1}{2}\Big(
  J'_{|I_n}(Y_{K,k})(Y-i_X Y) - A'_{k|I_n}(Y_{K,k})(Z_{K,k},Y-i_X Y) \Big)
\end{multline}
and the error discretization of the fast scale problem.
\begin{equation}\label{fs_subdivision}
  \eta_\EF^n := \frac{1}{2}\epsilon\int_{I_n} \eta^\pi\big(Y_{K,k}(t)\big)\cdot
  \big( Z(t)+Z_{K,k}(t)\big)\,\text{d}s.
\end{equation}

\FloatBarrier

\begin{figure}[t]
  \centering
  \includegraphics[width=0.48\textwidth]{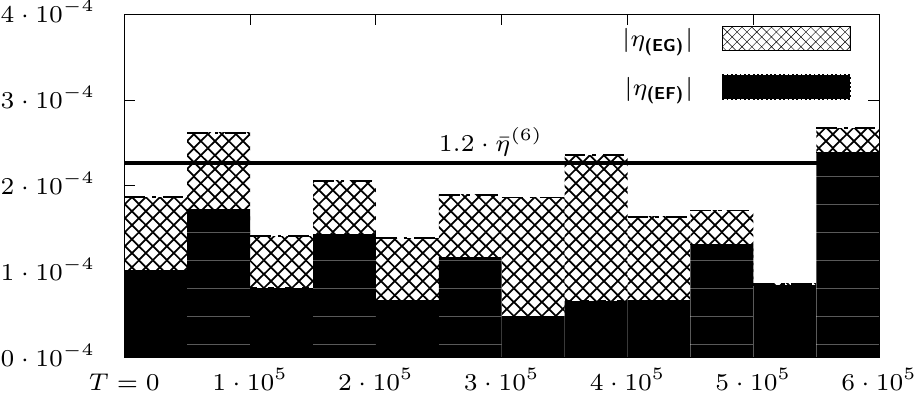}
  \includegraphics[width=0.48\textwidth]{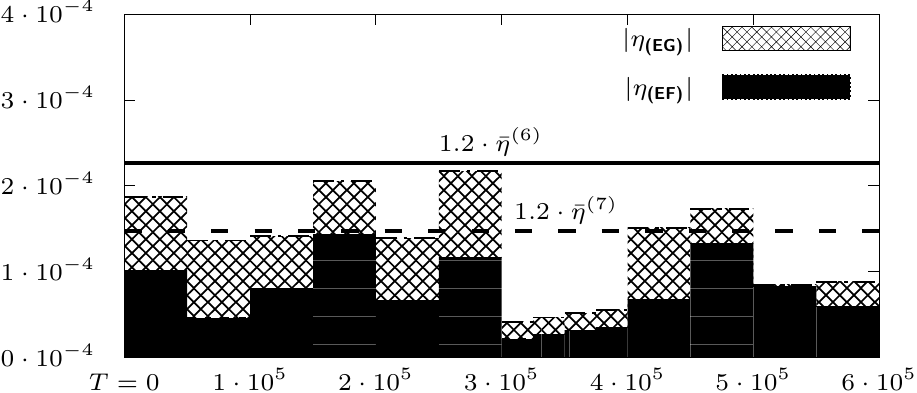}
  \caption{Absolute values of the error estimator $|\eta_\EG^n|+|\eta_\EF^n|$ for all macro steps before (left) and after (right) the sixth step of adaptive refinement. The bold line indicates the refinement threshold for iteration $6$, the dashed line indicates iteration 7.}
  \label{fig:step67}
\end{figure}

We introduce the following method for refining macro scale and micro scale:
\begin{algo}[Adaptive refinement]
  \label{algo:adaptivity}
  Let an initial subdivision $I_K^{(1)}$ into $N^{(1)}$ macro steps be given with uniform but possibly distinct partitions $I^{P,(1)}_n$ for each subdivision $n=1,\dots,N^{(l)}$. Let $\beta\in\mathds{R}$ with $\beta\approx 1$. Iterate for $l=1,2,\dots$
  \begin{enumerate}
  \item \label{algo:i1} Compute $\eta_\EG^n$ and $\eta_\EF^n$ for each $n=1,\dots,N^{(l)}$. 
  \item Calculate the average\vspace{-1em}
    \begin{equation}\label{threshold} \vspace{-1.em}
      \bar\eta^{(l)}\coloneqq  \frac{1}{N^{(l)}} \sum_{n=1}^{N^{(l)}} \Big( |\eta_\EG^n| + |\eta_\EF^n|  \Big)
    \end{equation}
  \item \label{algo:i2} For each $n=1,\dots,N^{(l)}$: if $|\eta_\EG^n| + |\eta_\EF^n| > \beta\cdot\bar\eta^{(l)}$, we refine this cell:
    \begin{enumerate}
    \item If $|\eta_\EG^n| > \beta |\eta_\EF^n|$ we refine $I_n = (T_{n-1}, T_n]$ into two intervals $(T_{n-1},T_n^*]$ and $(T_n^*,T_n]$ where $T_n^*$ is the midpoint of $I_n$. $I_n^{P,(l)}$ is kept for both new steps. 
        \item If $|\eta_\EF^n| > \beta |\eta_\EG^n|$ refine the subdivision $I_n^{P,(l)}$ by cutting the step size in half.
        \item Otherwise refine $I_n$ and $I^P_n$ according to \emph{3.a)} and \emph{3.b)}. 
    \end{enumerate}
  \end{enumerate}
\end{algo}
We illustrate the functionality of Algorithm~\ref{algo:adaptivity} starting with $I_K^{(1)}$ with $K=50\,000$ and $k=0.05$ on each $I_n^{P,(1)}$.

In Figure~\ref{fig:step67} we discuss the sixth refinement step of Algorithm~\ref{algo:adaptivity} in detail. The upper figure shows the error estimator $\eta_{K,k}^{(6)}$ and its partitioning into $\eta_\EG$ and $\eta_\EF$ for each of the 12 macro steps (there has been no refinement of $K$ in the first 5 iterations). The bold line indicates the tolerance for refinement, i.e. $\beta\cdot \bar\eta^{(6)}$ for $\beta=1.2$. Three steps exceed this limit and will be refined. In $I_2^{(6)}$ and $I_{12}^{(6)}$ the micro scale error is dominating and \emph{Step 3.b)} is applied, in $I_{8}^{(6)}$ the dominance of the macro scale error leads to a refinement on the $K$-scale according to \emph{Step 3.a)}. To keep the patch structure of the macro mesh we also refine $I_7^{(6)}$. The resulting discretization and the error estimator in the next step is shown in the lower plot.

The adaptive algorithm roughly balances the error contributions coming from macro error and micro error over the first couple of steps, see Fig.~\ref{fig:ratio_FS_LS} for details. 
In Fig.~\ref{fig:ratio_dwr_err} we further plot the effectivity index~(\ref{effectivity}) on this sequence of adaptively refined meshes and show that the error estimator still gains accuracy for increased resolution in $k$ and $K$.

Refinement in Algorithm~\ref{algo:adaptivity} is based on the absolute values of the local error contributions $|\eta_\EF^n|$ and $|\eta_\EG^n|$ and we introduce the indicator index \vspace{-0.5em}
\begin{equation}\label{indicators}
  \text{ind}_{K,k}\coloneqq
  \sum_{n=1}^N \frac{|\eta_\EG^n| + |\eta_\EF^n|}{|J(y_{ref})-J(Y_{K,k})|}. 
\end{equation}

Figs.~\ref{fig:ratio_thresh_err} and~\ref{fig:acc_thresh_parts} show values close to one and suggest no significant overestimation, neither in the complete error or in the single parts. 

\begin{figure}[t]
  \centering
  \begin{subfigure}[t]{0.45\textwidth}
    \caption{Effectivity of the error estimator.}
    \includegraphics[width=\textwidth]{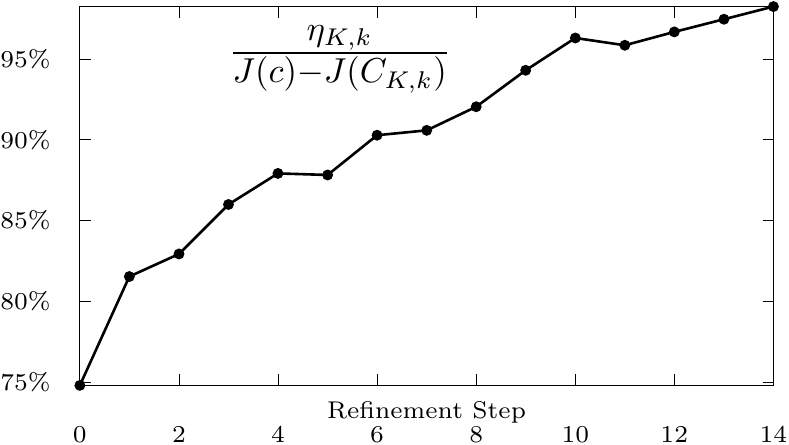}
    \label{fig:ratio_dwr_err}
  \end{subfigure}
  ~ 
  \begin{subfigure}[t]{0.45\textwidth}
    \caption{Effectivity of the error indicators. }
    \includegraphics[width=\textwidth]{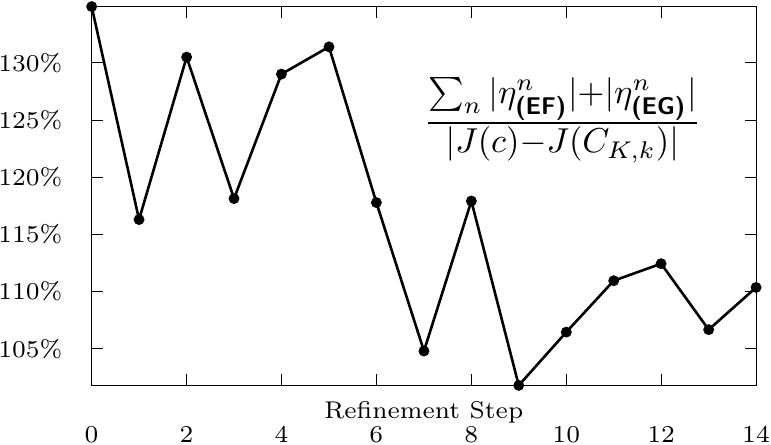}
    \label{fig:ratio_thresh_err}
  \end{subfigure}
  
  \begin{subfigure}[t]{0.45\textwidth}
    \caption{Balancing of micro and macro errors.}
    \includegraphics[width=\textwidth]{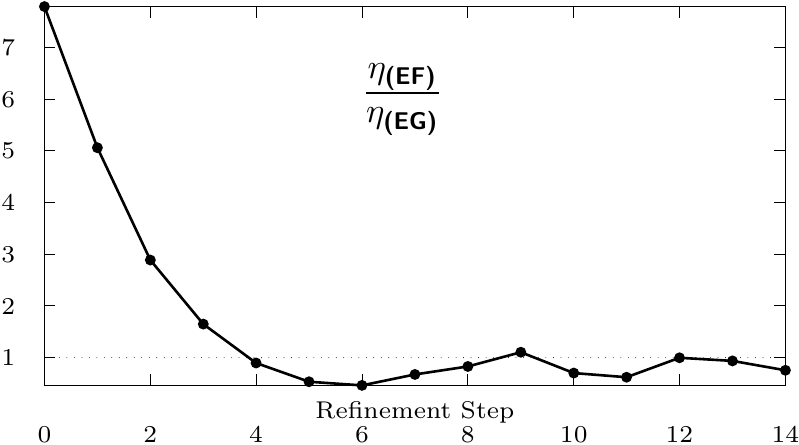}
    \label{fig:ratio_FS_LS}
  \end{subfigure}
  ~ 
  \begin{subfigure}[t]{0.45\textwidth}
    \caption{Effectivity of the partial indicators.}
    \includegraphics[width=\textwidth]{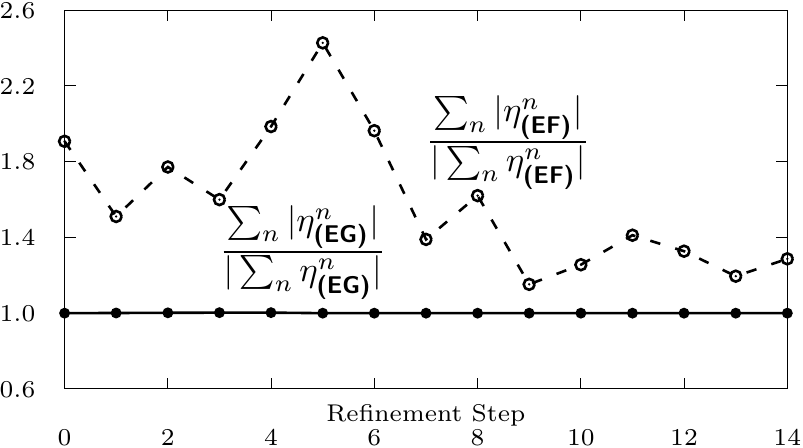}
    \label{fig:acc_thresh_parts}
  \end{subfigure}
  \caption{Performance on adaptive meshes: Effectivity~(\ref{effectivity}) (top/left), indicator-effectivity~(\ref{indicators}) (right) and trend towards balancing error contributions (bottom/right).}
\end{figure}

\begin{figure}[h!]
  \centering
  \begin{subfigure}[b]{0.45\textwidth}
    \caption{Error plotted over the effort~(\ref{effort_meas}).}
    \includegraphics[width=\textwidth]{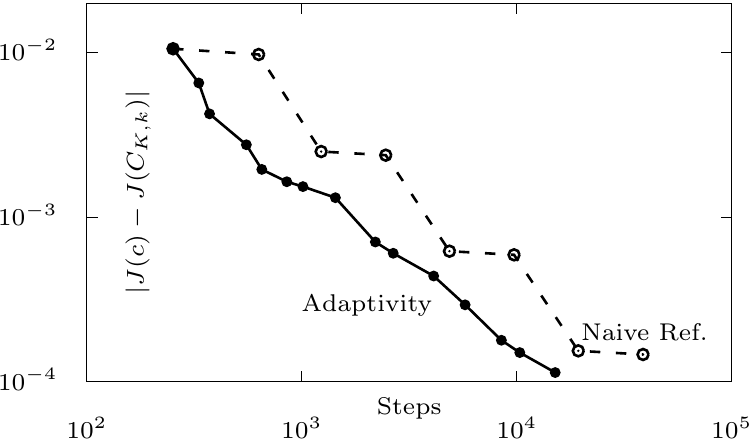}
    \label{fig:error_effort}
  \end{subfigure}
  ~ 
  \begin{subfigure}[b]{0.45\textwidth}
    \caption{Error over the cumulative effort. }
    \includegraphics[width=\textwidth]{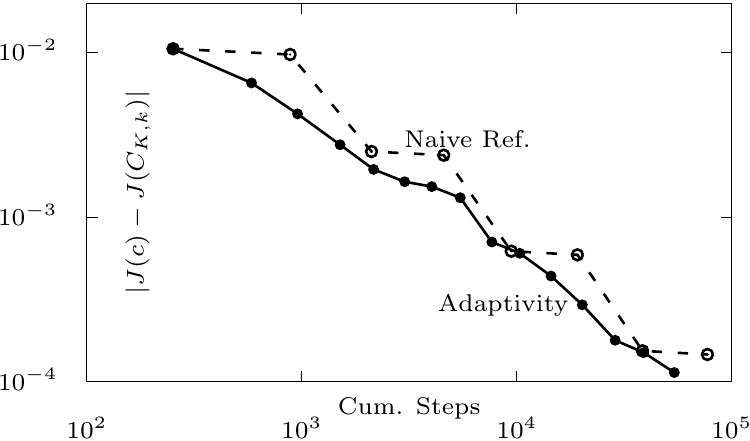}
    \label{fig:error_cumeffort}
  \end{subfigure}
  \caption{Comparison of the accuracy on uniform mesh refinement (alternately in micro and macro problem) and adaptive meshes. }\label{fig:adaptivityvsnaive}
\end{figure}

To measure the computational effort on locally refined discretizations we count the overall number of time steps to be computed in the macro and the micro problem:\vspace{-0.5em}
\begin{equation}\label{effort_meas}
  E_{K,k}^{(l)}\coloneqq \sum_{n=1}^{N^{(l)}} \Big( 1 + \frac{1}{k_n^{(l)}} \Big). 
\end{equation}
We do not take into account that multiple iterations are required within the Newton solver and that multiple cycles must be repeated for finding a periodic solution. The computation of the dual solution requires roughly the same effort, since the scheme runs backwards in time and also calls for the solution of periodic in time micro problem. With these solutions one can compute the error estimator. Figure~\ref{fig:adaptivityvsnaive} shows the error $J(y_{ref})-J(Y_{K,k})$ on sequences of adaptive and uniform meshes plotted over the effort~(\ref{effort_meas}). Adaptivity gives a slight advantage for the adaptive discretization. Since the regularity of the solution is very high, significant local effects cannot be expected.


\subsection{Test case with pronounced local behavior}
\label{sec:num2}


As a second test case we consider the following slightly modified problem that shows a more pronounced dependency of the slow scale on the fast scale solution. Here, we expect a larger benefit of local mesh adaptivity. Fig.~\ref{fig:macro2} shows the averaged slow scale solution  $Y(t)$ for both test cases. 

\begin{problem}\label{problem:numex2}
  On $I=[0,10^6]$ find $y\in C^1(I)$ and $u\in C^2(I)$ such that
  \[
    \begin{aligned}
      &y'(t) = \epsilon f\big(y(t),u(t)\big), &y(0)&=0,\\
      &u''(t) + \frac{1}{2} u'(t) + \gamma(y(t))u(t) = \sin\big(2\pi t\big),& u(0)=u_0,\, u'(0)&=u_0',
    \end{aligned}
  \]
  with the scale separation parameter $\epsilon=10^{-6}$ and
  \[
    f\big(y(t),u(t)\big) \coloneqq \frac{\tanh\big( 500u(t)^2 - 5 \big) +1.01}{1+y(t)},\quad 
      \gamma(y) \coloneqq 20 \tanh\big( -10y(t) + 6\big) +21.
  \]
\end{problem}
Similar to Lemma~\ref{lemma:problemass} we can show that this
problem also falls into the general framework discussed in this
paper. The proof follows that of Lemma~\ref{lemma:problemass} line by line.
\begin{lemma}
  Problem~\ref{problem:numex2} satisfies
  Assumptions~\ref{assumption:f} and~\ref{assumption:g}.
\end{lemma}

Again we produce reference values for the functional output
$J(y)=y(10^6)$ by resolved simulations based on a direct
discretization of Problem~\ref{problem:numex2} with the trapezoidal rule using  a small time step size $k$ over the full period of time
$I=[0,10^6]$. Extrapolating $k\to 0$ shows the experimental order of
convergence ${\cal O}(k^{2.0014})$ as expected and for further
comparisons we set the reference value to
\begin{equation}\label{referencevalue2}
  J(y_{ref})\coloneqq  0.59223654. 
\end{equation}

\begin{figure}[t]
  \centering
  \includegraphics[width=0.48\textwidth]{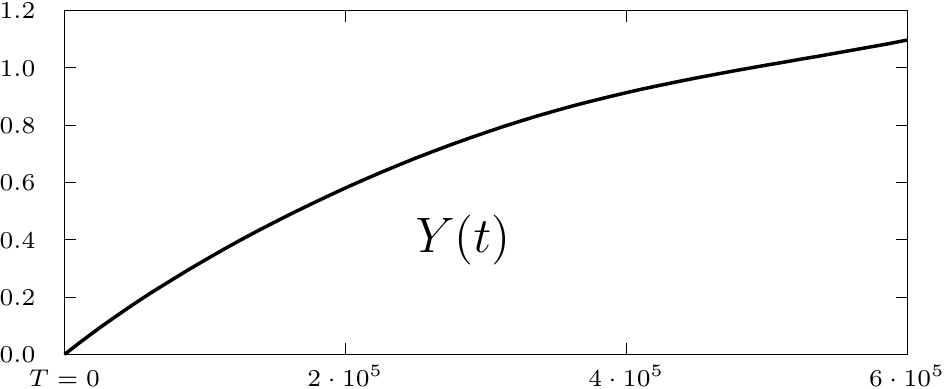}
  \includegraphics[width=0.48\textwidth]{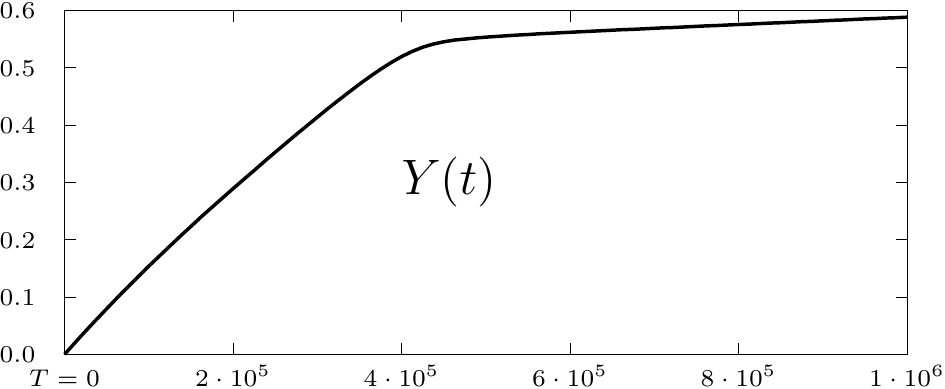}
  %
  %
  \caption{Solution $Y(t)$ of Problem~\ref{problem:numex1} (left) and Problem~\ref{problem:numex2} (right). }
  \label{fig:macro2}
\end{figure}

The solution to this problem alters its character at $t\approx 4.5\cdot 10^5$, see Fig.~\ref{fig:macro2}. Due to the sigmoid nature of $f(y,u)$ in the second argument we have a sudden decay in the rate of change of $Y$. The contributions to the error are also concentrated around this point of interest. For this case we illustrate the functionality of Algorithm~\ref{algo:adaptivity} starting with $I_K^{(1)}$ with $K=50\,000$ and $k=0.05$ on each $I_n^{P,(1)}$. 

\begin{figure}[t]
  \centering
  \includegraphics[width=0.48\textwidth]{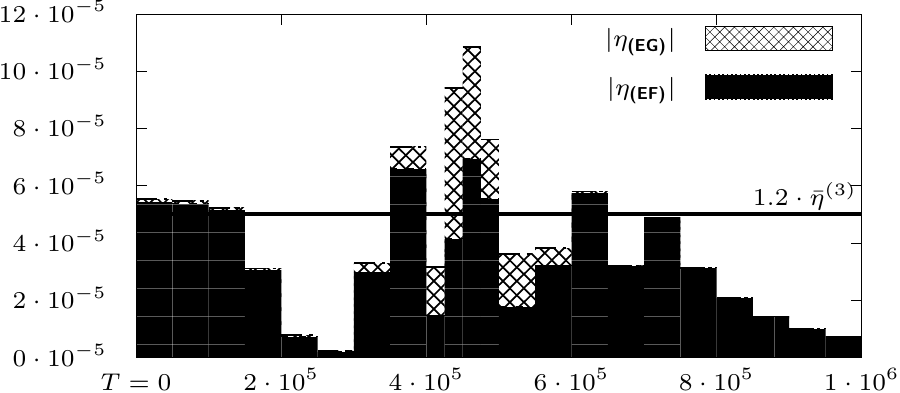}
  \includegraphics[width=0.48\textwidth]{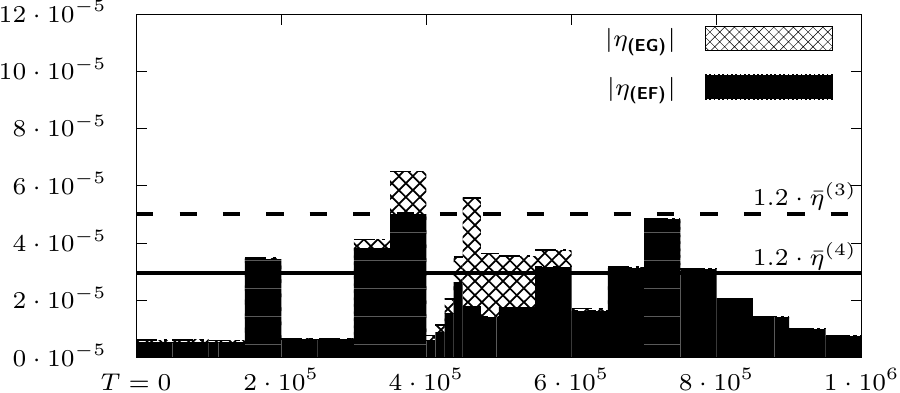}
  \caption{Absolute values of the error estimator $|\eta_\EG^n|+|\eta_\EF^n|$ for all macro steps before (left) and after (right) the third step of adaptive refinement. The bold line indicates the refinement threshold for iteration $3$, the dashed line indicates iteration $4$.}
  \label{fig:step34}
\end{figure}

In Fig.~\ref{fig:step34} the third refinement step is shown. In the upper graph we see the contributions to the error estimator from the subdivisions of $I$ after two refinement steps. The error still concentrates around the point of interest but due to the refinement of the adjacent patches, the estimator values are already better balanced. 
As in Fig.~\ref{fig:step67}, the bold line indicates the tolerance for refinement and for all subdivisions where the total contribution to the error estimator is above this threshold, either $k$ or $K$ is refined. The results of refinement are presented in the right graph. Comparing the bold line to the dotted line shows that there is an improvement in the value for the error.

\begin{figure}[t]
  \centering
  \begin{subfigure}[t]{0.45\textwidth}
    \caption{Balancing of micro and macro errors.}
    \includegraphics[width=\textwidth]{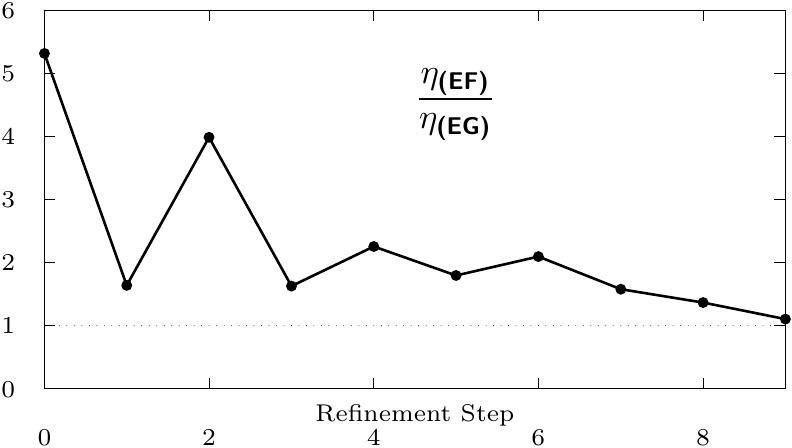}
    \label{fig:ratio_FS_LS2}
  \end{subfigure}
  ~ 
  \begin{subfigure}[t]{0.45\textwidth}
    \caption{Effectivity of the partial indicators.}
    \includegraphics[width=\textwidth]{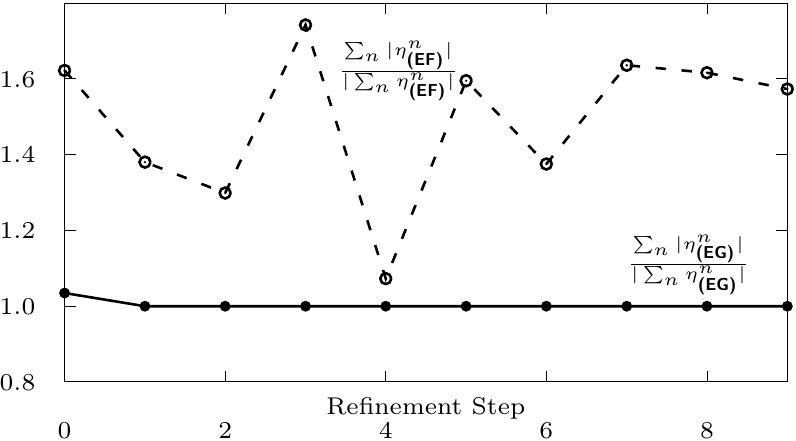}
    \label{fig:acc_thresh_parts2}
  \end{subfigure}
  \caption{Performance on adaptive meshes: Indicator-effectivity~(\ref{indicators}) (right) and trend towards balancing contributions (left).}
\end{figure}

We can see that the error parts $\eta_{\EG}$ and $\eta_{\EF}$ tend to balance, see Fig.~\ref{fig:ratio_FS_LS2}. This is intuitive since refinement is done on the dominant error term of a subdivision according to Algorithm~\ref{algo:adaptivity}. Figure~\ref{fig:acc_thresh_parts2} shows that the partial indicators reasonably approximate the error estimator in this case even though the sign of the estimator can and does change on different subdivisions and during refinement. This shows that contributions of one sign dominate in this case but this is not always guaranteed. In the rare case of the positive and negative contributions to the error estimator cancelling each other out, the error estimator can fail to give an accurate representation of the error.

\begin{figure}[h!]
  \centering
  \begin{subfigure}[b]{0.45\textwidth}
    \includegraphics[width=\textwidth]{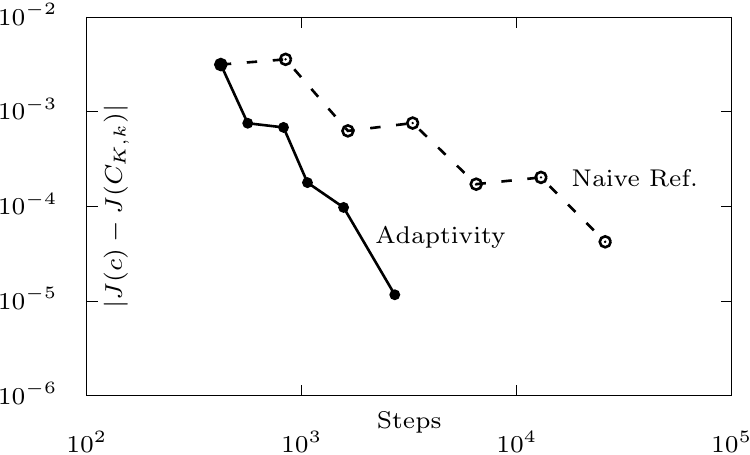}
    \caption{Error plotted over the effort~(\ref{effort_meas}).}
    \label{fig:error_effort2}
  \end{subfigure}
  ~ 
  \begin{subfigure}[b]{0.45\textwidth}
    \includegraphics[width=\textwidth]{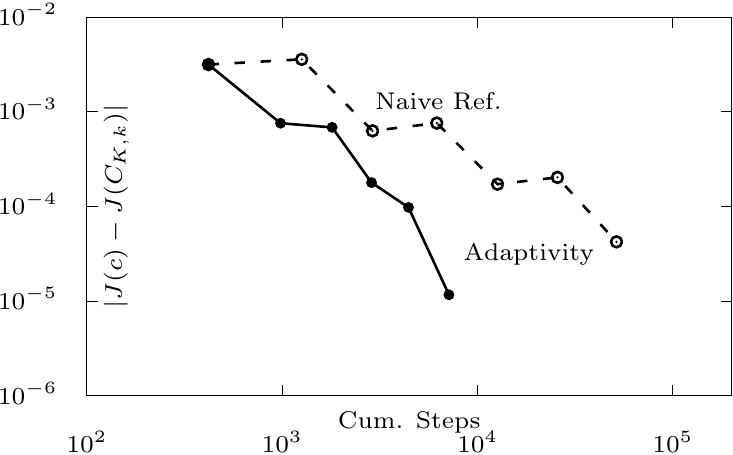}
    \caption{Error over the cumulative effort. }
    \label{fig:error_cumeffort2}
  \end{subfigure}
  \caption{Comparison of the accuracy on uniform and adaptive meshes.}
  \label{fig:adaptivityvsnaive2}
\end{figure}

For this problem our algorithm speeds up performance significantly. According to Fig.~\ref{fig:error_effort2}, to achieve a similar error of around $5\cdot10^{-5}$, we need only around one tenth of the effort. The cumulative effort, see Fig.~\ref{fig:error_cumeffort2}, is also of similar magnitude. This can be explained by the unequal distribution of $\eta_{\EG}$ and especially $\eta_{\EF}$, see Fig.~\ref{fig:step34}. By uniform refinement of $I$ and $I^P$ we cannot alleviate this phenomenon. However our error estimator is able to identify sections of $I$ that need to be treated on a finer scale and also sections, where $I^P$ needs to be subdivided with a smaller $k$ compared to other sections. Local refinement then gives a better distribution of the computing resources to calculate the solution $Y(t)$. Since the dependency on the slow scale is not uniform in this problem, we are able to gain efficiency with adaptive refinement unlike in the first numerical example. 

To conclude we compare in Table~\ref{tab:alleffort} the effort,
measured in time-steps of the micro-problem to be solved, for the
fully resolved simulation of the original problem with the multiscale
approach based on uniform meshes and on adaptive meshes. To measure
the effort we must multiply the cumulative effort listed above by 2,
to account for the additional effort for solving the adjoint problem
and also by 5, which accounts for the number of cycles required to
find the periodic state. In each case we show the results for the
choice of discretization parameters where the error is below
$5\cdot 10^{-5}$. By using the multiscale scheme the effort is reduced by
about $1:400$, although we must solve adjoint solutions and although 
the approximation of the local periodic solutions requires multiple
cycles in each macro step. If we further employ adaptive mesh
refinement the effort is reduced by $1:2\,800$. The adaptive
multiscale scheme has the further advantage of giving an estimate on
the discretization error. 

\begin{table}[t]
  \begin{center}
    \begin{tabular}{lcccr}
      \toprule
      Approach & Error & $k$ & $K$ & Micro-steps \\
      \midrule 
      Resolved simulation  &$3.43\cdot 10^{-5}$&$\frac{1}{200}$&--&200\,000\,000\\
      Multiscale (uniform) &$4.24\cdot 10^{-5}$&$\frac{1}{160}$&6\,250&513\,800\\
      Multiscale (adaptive)&$1.16\cdot 10^{-5}$&$\frac{1}{320}$&12\,500&71\,360 \\
      \bottomrule
    \end{tabular}
  \end{center}
  \caption{Comparison of the effort for the resolved simulation, the
    multiscale scheme based on uniform and adaptive discretizations measured in the number of overall steps of the micro problem to be  solved.}
  \label{tab:alleffort}
\end{table}


\section{Conclusion}
  We have presented an a posteriori error estimator for a temporal multiscale scheme of HMM type that has recently been introduced~\cite{FreiRichter2020}. This multiscale scheme is based on separating micro and macro scale by replacing the micro scale influences by localized periodic in time solutions. The resulting scheme calls for the solution of one such periodic micro problem in each macro step.

The error estimator is based on the dual weighted residual method for estimating errors in goal functionals. The adjoint problem entering the error estimator has a structure similar to the primal one: each adjoint macro time step requires the solution of a periodic micro problem. In addition, to incorporate the error of the periodic in time micro scale problems, a further adjoint micro problem must be solved in each macro step. The resulting error estimator allows for a splitting of the local error contributions into micro scale and macro scale influences. We have shown very good efficiency of the estimator for a wide range of discretization parameters.

Based on the splitting into micro scale errors and macro scale errors an adaptive refinement loop is presented that allows to optimally balance all discretization parameters.

For the future it remains to extend this setting to temporal multiscale problems involving partial differential equations as discussed in~\cite{FreiRichter2020,MizerskiRichter2020} that will add the further complexity of finding optimal spatial discretization parameters for macro and micro problems.

\paragraph{Acknowledgements}
The work of both authors has been supported by the Deutsche Forschungsgemeinschaft (DFG, German Research Foundation) as part of the GRK 2297 MathCore - 314838170, as well within the project 411046898. Further, the work of TR has been supported by the Federal Ministry of Education and Research of Germany, grant number 05M16NMA.

    \bibliographystyle{plain}
    \bibliography{lit}

\end{document}